\documentclass[12pt]{amsart}       
\usepackage{txfonts}
\usepackage{amssymb}
\usepackage{eucal}
\usepackage{graphicx}
\usepackage{amsmath}
\usepackage{mathrsfs}
\usepackage{amscd}
\usepackage[all]{xy}           
\usepackage{amsfonts,latexsym}
\usepackage{xspace}
\usepackage{epsfig}
\usepackage{float}
\usepackage{color}
\usepackage{fancybox}
\usepackage{colordvi}
\usepackage{multicol}
\usepackage{colordvi}
\usepackage{tikz}
\usepackage{wasysym}
\usepackage[active]{srcltx} 
\ifpdf
  \usepackage[colorlinks,final,backref=page,hyperindex]{hyperref}
\else
  \usepackage[colorlinks,final,backref=page,hyperindex,hypertex]{hyperref}
\fi
\usepackage{enumerate}

\newcommand{\nc}{\newcommand}
\newcommand{\delete}[1]{}	
\nc{\mlabel}[1]{\label{#1}}  
\nc{\mcite}[1]{\cite{#1}}  
\nc{\mref}[1]{\ref{#1}}  
\nc{\meqref}[1]{\eqref{#1}} %
\nc{\mbibitem}[1]{\bibitem{#1}} 

\delete{
\nc{\mlabel}[1]{\label{#1}  
{\hfill \hspace{1cm}{\small\tt{{\ }\hfill(#1)}}}}
\nc{\mcite}[1]{\cite{#1}{\small{\tt{{\ }(#1)}}}}  
\nc{\mref}[1]{\ref{#1}{{\tt{{\ }(#1)}}}}  
\nc{\meqref}[1]{\eqref{#1}{{\tt{{\ }(#1)}}}}  
\nc{\mbibitem}[1]{\bibitem[\bf #1]{#1}} 
}

\newtheorem{theorem}{Theorem}[section]
\newtheorem{prop}[theorem]{Proposition}
\newtheorem{lemma}[theorem]{Lemma}
\newtheorem{coro}[theorem]{Corollary}

\theoremstyle{definition}
\newtheorem{defn}[theorem]{Definition}
\newtheorem{prop-def}{Proposition-Definition}[section]

\newtheorem{notation}[theorem]{Notation}

\newtheorem{tempex}[theorem]{Example}
\newtheorem{tempexs}[theorem]{Examples}
\newtheorem{temprmk}[theorem]{Remark}
\newtheorem{tempexer}{Exercise}[section]
\newenvironment{exam}{\begin{tempex}\rm}{\end{tempex}}

\nc{\vsa}{\vspace{-.1cm}} \nc{\vsb}{\vspace{-.2cm}}
\nc{\vsc}{\vspace{-.3cm}} \nc{\vsd}{\vspace{-.4cm}}
\nc{\vse}{\vspace{-.5cm}}

\nc{\tred}[1]{\textcolor{red}{#1}} \nc{\tgreen}[1]{\textcolor{green}{#1}}
\nc{\tblue}[1]{\textcolor{blue}{#1}} \nc{\tpurple}[1]{\textcolor{purple}{#1}}

\nc{\li}[1]{\textcolor{red}{#1}}
\nc{\lir}[1]{\textcolor{red}{Li:#1}}
\nc{\xing}[1]{\tblue{Xing:#1 }}
\nc{\hu}[1]{\tpurple{Huhu:#1 }}
\nc{\Hu}[1]{\tpurple{#1 }}

\nc{\ubqco}{unary binary \qc ns operad\xspace}
\nc{\ubqcos}{unary binary \qc ns operads\xspace}
\nc{\Ubqco}{Unary binary \qc ns operad\xspace}
\nc{\Ubqcos}{Unary binary \qc ns operads\xspace}

\nc{\multa}{\bullet_1}\nc{\multb}{\bullet_2}
\nc{\multi}{\bullet_i}\nc{\multj}{\bullet_j}
\nc{\opera}{P_1}\nc{\operb}{P_2}
\nc{\operi}{P_i}\nc{\operj}{P_j}

\nc{\name}[1]{{\bf #1}}
\nc{\bfc}{\mathbf{c}}

\nc\sopr[1]{{#1}}

\nc{\mscr}[1]{\mathscr{#1}} \nc{\cal}[1]{\mathcal{#1}} \nc{\bb}[1]{\mathbb{#1}}
\nc{\id}{\rm{id}} \nc{\bfk}{{\bf k}}
\nc{\rba}{\mathscr{RBA}}\nc{\spp}{\mathscr{P}} \nc{\stt}{\mscr{T}}
\nc{\homop}{\mathscr{H}\!om\mathscr{A}} \nc{\cubas}{\mathscr{CA}s} \nc{\spb}{\mscr{P}_\circ}  \nc{\sph}{\mscr{P}_\bullet}
\nc{\dera}{\mathscr{D}er\mathscr{A}}
\nc{\add}{\uplus}\nc{\badd}{\biguplus}
\nc{\lp}{\lambda_p}\nc{\la}{\lambda_a}
\nc{\prl}[2]{(x#1y)#2z}\nc{\prr}[2]{x#1(y#2z)}
\nc{\ra}[3]{#1(x)#2#3(y)}\nc{\rb}[3]{#3(#1(x)#2y)+#1(x#2#3(y))}\nc{\rbwe}[3]{#3(#1(x)#2y)+#1(x#2#3(y))+\lambda_{#3}#1(x#2y)}
\nc{\dereq}[6]
{\treey{\cdlr[0.8]{o} \cdl{ol}\cdr{or}#1{o/b}
\node at (0.2,-0.2) {$d$};\cdlr{o}\node  at (0,0) {$#2$};}
-\treey{\cdlr[0.8]{o} \cdl{ol}\cdr{or}
\node at (ol) {$#4$};\node at (-0.2,0.5) {$d$};\cdlr{o}\node  at (0,0) {$#3$};}
-\treey{\cdlr[0.8]{o} \cdl{ol}\cdr{or}
\node at (or) {$#6$};\node at (0.2,0.5) {$d$};\cdlr{o}\node  at (0,0) {$#5$};}}
\nc{\kb}[3]{\beta_{#1,#2}^{#3,n}}
\nc{\kc}[4]{\beta_{#1,#2,#3}^{#4,n}}
\nc{\ka}[3]{\gamma_{#1,#2}^{#3,n}}
\nc{\kab}[4]{\gamma_{#1,#2,#3}^{#4,n}}
\nc{\kd}[4]{\kappa_{#1,#2, #3}^{#4,n}}
\nc{\otimesh}{\mathop{\otimes}_{\text{H}}}%
\nc{\kong}{\noindent}
\nc{\rblineq}[5]{\Big(
\treey{\cdlr[0.8]{o}\cdl{ol}\cdr{or}
\node at (ol) {$#1$};\node at (or) {$#3$};\node at (-0.2,0.5) {$#2$};\node at (0.2,0.5) {$#4$};\node at (0,0) {$#5$};}
-\treey{\cdlr[0.8]{o} \cdl{ol}\cdr{or}\node at (0,0) {$#5$};\node at (0,-0.2) {$#3$};
\node at (ol) {$#1$};\node at (-0.2,0.5) {${#2}$};\node at (0.2,-0.2) {$#4$};}
-\treey{\cdlr[0.8]{o} \cdl{ol}\cdr{or}\node at (0,0) {$#5$};\node at (0,-0.2) {$#1$};
\node at (or) {$#3$};\node at (0.2,-0.2) {$#2$};\node at (0.2,0.5) {$#4$};}\Big)}
\nc{\rblineqq}[5]{
\treey{\cdlr[0.8]{o}\cdl{ol}\cdr{or}
\node at (ol) {$#1$};\node at (or) {$#3$};\node at (-0.2,0.5) {$#2$};\node at (0.2,0.5) {$#4$};\node at (0,0) {$#5$};}
-\treey{\cdlr[0.8]{o} \cdl{ol}\cdr{or}\node at (0,0) {$#5$};\node at (0,-0.2) {$#3$};
\node at (ol) {$#1$};\node at (-0.2,0.5) {${#2}$};\node at (0.2,-0.2) {$#4$};}
-\treey{\cdlr[0.8]{o} \cdl{ol}\cdr{or}\node at (0,0) {$#5$};\node at (0,-0.2) {$#1$};
\node at (or) {$#3$};\node at (0.2,-0.2) {$#2$};\node at (0.2,0.5) {$#4$};}}

\usetikzlibrary{calc}
\newlength\xch\newlength\dbj
\newif\ifqdd
\nc\cddf[3]{%
\coordinate (#2) at ($(#1)+(#3)$);
\draw (#1)--(#2);
\ifqdd\fill (#1) circle (\dbj);\fi}
\nc\cdx[4][1]{\cddf{#2}{#3}{#4:#1*\xch}}
\nc\cdu[2][1]{\cdx[#1]{#2}{#2u}{90}}
\nc\cdl[2][1]{\cdx[#1]{#2}{#2l}{135}}
\nc\cdr[2][1]{\cdx[#1]{#2}{#2r}{45}}
\nc\cdlr[2][1]{%
\foreach \i in {#2} {\cdl[#1]{\i}\cdr[#1]{\i}}}
\nc\cduu[2][1]{%
\foreach \i in {#2} {\cdu[#1]{\i}\cdu[#1]{\i}}}
\nc\cda[2][1]{\cdx[#1]{#2}{#2a}{90}}
\nc\cdb[2][1]{\cdx[#1]{#2}{#2b}{-90}}
\nc\cdbl[2][1]{\cdx[#1]{#2}{#2bl}{-135}}
\nc\cdbr[2][1]{\cdx[#1]{#2}{#2br}{-45}}
\nc\cdblr[2][1]{%
\foreach \i in {#2} {\cdbl[#1]{\i}\cdbr[#1]{\i}}}
\nc\treeo[2][]{\tikz[baseline=-0.58ex,line width=0.15ex,
every node/.style={font=\scriptsize,inner sep=1pt},#1]{%
\coordinate (o) at (0,0);#2}}%
\nc\treeyy[2][scale=0.8]{\treeo[#1]{\cdb o\cda o#2}}%
\nc\treey[2][scale=0.8]{\treeo[#1]{\cdb o\cdlr o#2}}%
\nc\treedy[2][scale=0.8]{\treeo[#1]{\cda o\cdblr o#2}}%
\nc\zhongdian[2]{ \node at($(#1)!0.5!(#2)$) {$\bullet$};}%
\nc\zhd[1]{\foreach \i/\j in {#1} {\zhongdian{\i}{\i\j}}}
\nc\zhongddian[2]{ \node at($(#1)!0.5!(#2)$) {$\circ$};}%
\nc\zhdd[1]{\foreach \i/\j in {#1} {\zhongddian{\i}{\i\j}}}
\nc\zhongsdian[2]{ \node at($(#1)!0.5!(#2)$) [scale=0.6]{$\bullet$};}%
\nc\zhds[1]{\foreach \i/\j in {#1} {\zhongsdian{\i}{\i\j}}}
\nc{\rc}{R_{\bfc,\,}}
\nc{\cubicr}[4]{\kc{k}{\ell}{i}{1}\treey{\cdlr[0.8]{o}\cdl{ol}\cdr{or}
\node at (ol) {$#1$};\node at (or) {$#2$};\node at (-0.2,0.5) {$P_k$};\node at (0.2,0.5) {$P_\ell$};\node at (0,0.2) {$i$};\node at (0,0) {$#3$};}
+\kc{k}{\ell}{i}{2}\treey{\cdlr[0.8]{o} \cdl{ol}\cdr{or}#4{o/b}
\node at (ol) {$#2$};\node at (-0.2,0.5) {$P_\ell$};\node at (0.2,-0.2) {$P_k$};\node at (0,0.2) {$i$};\node at (0,0) {$#3$};}
+\kc{k}{\ell}{i}{3}\treey{\cdlr[0.8]{o} \cdl{ol}\cdr{or}#4{o/b}
\node at (or) {$#2$};\node at (0.2,-0.2) {$P_k$};\node at (0.2,0.5) {$P_\ell$};\node at (0,0.2) {$i$};\node at (0,0) {$#3$};}
+\kc{k}{\ell}{i}{4}\treey{\cdlr[0.8]{o} \cdl{ol}\cdr{or}\node at (0,-0.35) {$#2$};
\node at (0,-0.15) {$#1$};\node at (0.2,-0.1) {$P_k$};\node at (0.2,-0.4) {$P_\ell$};\node at (0,0.2) {$i$};\node at (0,0) {$#3$};}}
\nc{\cubicrr}[4]{\kc{k}{\ell}{i}{1}\treey{\cdlr[0.8]{o}\cdl{ol}\cdr{or}
\node at (ol) {$#1$};\node at (or) {$#2$};\node at (-0.2,0.5) {$P_k$};\node at (0.2,0.5) {$P_\ell$};\node at (0,0.2) {$i$};\node at (0,0) {$#3$};}
+\kc{k}{\ell}{i}{2}\treey{\cdlr[0.8]{o} \cdl{ol}\cdr{or}#4{o/b}
\node at (ol) {$#2$};\node at (-0.2,0.5) {$P_\ell$};\node at (0.2,-0.2) {$P_k$};\node at (0,0.2) {$i$};\node at (0,0) {$#3$};}\\
+&\kc{k}{\ell}{i}{3}\treey{\cdlr[0.8]{o} \cdl{ol}\cdr{or}#4{o/b}
\node at (or) {$#2$};\node at (0.2,-0.2) {$P_k$};\node at (0.2,0.5) {$P_\ell$};\node at (0,0.2) {$i$};\node at (0,0) {$#3$};}
+\kc{k}{\ell}{i}{4}\treey{\cdlr[0.8]{o} \cdl{ol}\cdr{or}\node at (0,-0.35) {$#2$};
\node at (0,-0.15) {$#1$};\node at (0.2,-0.1) {$P_k$};\node at (0.2,-0.4) {$P_\ell$};\node at (0,0.2) {$i$};\node at (0,0) {$#3$};}}
\nc{\qc}{quadratic/cubic }
\nc{\lin}[1]{{#1}^{\mathrm{LC}}} 	
\nc{\Lin}[1]{{#1}_{\mathrm{LC}}}
\nc{\mat}[1]{{#1}^{\mathrm{MT}}} \nc{\Mat}[1]{{#1}_{\mathrm{MT}}}
\nc{\tot}[1]{{#1}^{\mathrm{TC}}} \nc{\Tot}[1]{{#1}_{\mathrm{TC}}}
\nc{\lrr}{R_{L}}\nc{\mrr}{R_{M}}\nc{\trr}{R_{T}}
\nc{\lrrs}{R'_{L}}\nc{\mrrs}{R'_{M}}\nc{\trrs}{R'_{T}}
\nc{\lr}[1]{#1_{L}} \nc{\mr}[1]{#1_{M}} \nc{\tr}[1]{#1_{T}}
\nc{\tra}[1]{#1_{T,1,2}}\nc{\traa}[1]{#1_{T,1,3}} \nc{\trb}[1]{#1_{T,2,3}} \nc{\trbb}[1]{#1_{T,2,2}} \nc{\trc}[1]{#1_{T,3,2}}\nc{\trcc}[1]{#1_{T,3,3}}\nc{\trd}[1]{#1_{T,4}}\nc{\tri}[1]{#1_{T,i,2}}\nc{\trii}[1]{#1_{T,i,3}}
\nc{\tru}[1]{#1_{T'}}
\nc{\barr}[1]{\bar{#1}}
\nc\dual{\ast}\nc\dkc[3]{\beta'^{b,#3}_{#1,#2}}
\nc{\linop}{\mscr{T}\Big(\bigoplus\limits_{\omega\in\Omega} E_\omega\Big)\Big/\Big\langle \bigcup\limits_{\omega\in\Omega}R_\omega\cup \lrr\Big\rangle}
\nc\des{{\rm Des_{\leq_{L}}}}\nc\ver{{\rm Vert}}\nc\dess{{\rm Des_{\leq}}}
\nc\Omegac{{\Omega^{\rm c}}}\nc{\Rbb}{R_{2,2}}\nc{\Rbc}{R_{2,3}}\nc{\Rbbw}{R_{\omega,2,2}}\nc{\Rbcw}{R_{\omega,2,3}}
\nc\pair[1]{{\langle#1\rangle}} \nc\as{\mscr{A}\hspace{-0.1cm}s} \nc\ass{\mscr{A}\hspace{-0.1cm}ss}
\nc\na{\mscr{N}\hspace{-0.1cm}\mscr{A}}\nc\dend{\mscr{D}\hspace{-0.05cm}end}\nc\dias{\mscr{D}\hspace{-0.05cm}ias}
\nc\mnodis{\mu,\nu,\omega\in\Omega \atop \text{distinct}}
\begin{document}

\title[Compatible structures on unary binary quadratic/cubic ns operads]{Compatible structures on unary binary nonsymmetric operads with quadratic and cubic relations}

\author{Xing Gao}
\address{School of Mathematics and Statistics,
Lanzhou University, Lanzhou, Gansu 730000, China}
\email{gaoxing@lzu.edu.cn}

\author{Li Guo}
\address{
Department of Mathematics and Computer Science,
Rutgers University,
Newark, NJ 07102, USA}
\email{liguo@rutgers.edu}

\author{Huhu Zhang}
\address{School of Mathematics and Statistics,
Lanzhou University, Lanzhou, Gansu 730000, China}
\email{zhanghh17@lzu.edu.cn}

\date{\today}

\begin{abstract}
Various compatibility conditions among replicated copies of operations in a given algebraic structure have appeared in many contexts in recent years. Taking an uniform approach, this paper gives an operadic study of compatibility conditions for nonsymmetric operads with unary and binary operations, and homogeneous quadratic and cubic relations. This generalizes the previous studies for binary quadratic operads. We consider three compatibility conditions, namely the linear compatibility, matching compatibility and total compatibility, with increasingly strict restraints among the replicated copies. The linear compatibility is in Koszul dual to the total compatibility, while the matching compatibility is self dual. Further, each compatibility can be expressed in terms of either one or both of the two Manin square products.
\end{abstract}

\subjclass[2010]{
    18D50,  
	05C05,   
	05A05,   
	16W99 
}

\keywords{operad; linear compatibility; matching compatibility; total compatibility; Manin product; Koszul duality; differential algebra; Rota-Baxter algebra}

\maketitle

\vspace{-1cm}

\tableofcontents

\vspace{-1cm}

\setcounter{section}{0}

\allowdisplaybreaks

\section{Introduction}
\mlabel{sec:intro}
This paper studies nonsymmetric operads encoding algebraic structures with replicated copies of operations satisfying various compatibility conditions among these copies.
The relations of the compatibility conditions with Koszul duality and Manin products are established.  

\newpage 
\subsection{Algebraic structures with replicated operations}
\subsubsection{Linear compatibility of operations}\vspace{-.3cm}
An important property of derivations and its multi-dimensional generalizations such as tangent vectors is their closure under taking linear combinations, allowing them to form a vector space and further a Lie algebra. Such properties cannot be expected for other operators or operations. Thus it is natural to investigate the conditions under which linear combinations of multiple copies of a given operation still have the same properties of this operation. Such a condition is called a {\bf linear compatibility condition}.

The notion of linear compatibility first appeared for the Lie bracket, arising from the pioneering work~\mcite{Mag} of Magri in the study of the integrable Hamiltonian equation.
There a bi-Hamiltonian system could be defined by a Poisson algebra with two linearly compatible Poisson (Lie) brackets.
Such a structure was abstracted to the notion of a bihamiltonian algebra and studied in the context of operads and Koszul duality~\mcite{DK}, and further applied to quadratic algebras in the sense of Manin and a conjecture of Feigin~\mcite{BDK}.
The Koszul property of the linearly compatible Lie operad was verified in~\mcite{DK2} applying posets of weighted partitions. This approach was further extended in~\mcite{DL,GW} to the study of free linearly compatible Lie algebras with multiple brackets.

Linear compatibility is naturally related to linear or infinitesimal deformations, studying when a linear turbulence $\mu+a \nu$ of an operation $\mu$ is still the same kind of operation. In this direction, deformations of bihamiltonian algebras were developed in~\mcite{LZ} and continued in~\mcite{CPS} where a vanishing conjecture raised in~\mcite{LZ} was proved applying spectral sequences.
\vspace{-.3cm}
\subsubsection{Compatibilities of binary quadratic operads and nonsymmetric operads}

Linear compatibility of the Lie operad was generalized to binary quadratic operads by~Strohmayer~\mcite{St} where it was also shown that linear compatibility is in Koszul dual to another naturally defined total compatibility,
in the sense that the linear compatibility of an operad has its Koszul dual as the total compatibility of the dual operad.

In the remarkable work of Bruned, Hairer and Zambotti~\mcite{BHZ} on algebraic
renormalization of regularity structures, another compatible condition for pre-Lie algebra with multiple operations emerged~\mcite{Foi}, called multiple pre-Lie algebras. As it turns out, its associative analogue with two multiplications was introduced in~\mcite{Zi} as $As^{(2)}$ and further studied in~\mcite{ZBG} under the name of matching associative dialgebras. In~\mcite{ZGG20}, such matching conditions were systematically studied for Rota-Baxter algebras, dendriform algebras and pre-Lie algebras.

Linearly compatible associative algebras were studied in~\mcite{OS1} for matrix algebras and especially for linear deformations. It was further showed in~\mcite{OS2} that a pair of linear compatible associative products gives rise to a hierarchy of integrable systems of ODEs via the Lenard-Magri scheme~\mcite{Mag}.
Double constructions of linearly compatible associative algebras have been studied further in the direction of compatible associative bialgebras, associative Yang-Baxter equations and Goncharov's path Hopf algebras~\mcite{Go,Mar,Wu}.

In~\mcite{CGM}, quantum bi-Hamiltonian systems were built on linearly compatible associative algebras.
In~\mcite{Dot}, linear compatible associative algebras were studied as $S_n$-modules and free objects were constructed in terms of rooted trees and grafting, and further related to the Hopf algebras of Connes-Kreimer, Grossman-Larson and Loday-Ronco.
Homotopy linear compatible algebras were introduced and the homotopy transfer theorem were proved in~\mcite{Zh}.

Totally compatible associative algebras and Lie algebra with two multiplications were further studied in~\mcite{ZBG2} in connection with tridendriform algebras and PostLie algebras.
\vspace{-.3cm}
\subsubsection{Compatibilities of Rota-Baxter operators}

More recently, compatible unary operations, that is, linear operators, have also been studied~\mcite{ZGG20} under the name of matching Rota-Baxter operators. Such studies were motivated on the one hand by imposing to Rota-Baxter operators the linear closeness of the derivations noted at the very beginning and, on the other hand, by the matching pre-Lie algebras arising from the work~\mcite{BHZ} on algebraic renormalization of regularity structures. Furthermore, such structure underlies the algebraic study of Voterra integral equations~\mcite{GGZy,GGL}.

From the viewpoint of deformation theory, while deformation theory for algebraic structures with binary operations is quite general and complete, as deformation of binary (quadratic) operads~\mcite{LV}, its study when the algebraic structure has nontrivial unary operations is experiencing rapid developments quite recently. See~\mcite{Da,LST,TBGS} and the references therein.
\vspace{-.1cm}
\subsection{Compatibility conditions of unary binary operads}
Motivated by these recent developments of algebraic structures with compatible unary and binary operations that satisfy quadratic or cubic relations, we give a systematic study of operads with compatible operations, generalizing the existing treatments of binary quadratic operads in several directions. Thus our approach
\begin{enumerate}
\item includes unary as well as binary operations,
\item examines cubic as well as quadratic relations, and
\item relates several compatibility conditions by Koszul duality and Manin products.
\end{enumerate}

In order to give a systematic treatment, we will focus on nonsymmetric (ns) operads in this paper and leave a detailed discussion of the other cases to a later study.

The main part of the paper is organized in three sections for three compatibility conditions.

In Section~\mref{sec:lincom}, we set up the stage of our study, on an \ubqco $\spp$, and provide many examples.
We then introduce the general structure of linearly compatible algebras over such an operad, for a parameter set $\Omega$. We then define the linearly compatible operad $\lin{\spp}_\Omega$ of $\spp$ that encodes the linearly compatible algebras ove $\spp$ (Theorem~\mref{thm:comp}). It is shown that $\lin{\spp}_\Omega$ is the Manin black square product of $\spp$ with the linearly compatible operad $\lin{\as}_\Omega$ of the associative algebra (Proposition~\mref{prop:maninbl}).

In Section~\mref{sec:matcom}, the matching compatible operad $\mat{\spp}_\Omega$ of $\spp$ is introduced and its self duality for the Koszul dual is proved (Theorem~\mref{thm:mdul}). Examples of self-dual operads with nontrivial unary operations are provided. Further, $\mat{\spp}_\Omega$ can be obtained from $\spp$ by taking its Manin black square product and white square product with $\mat{\as}_\Omega$, the matching associative operad (Proposition~\mref{prop:maninbll}).

The notion of totally compatible operad $\tot{\spp}_\Omega$ is introduced in Section~\mref{sec:totcom} and its duality with $\lin{(\spp^!)}_\Omega$ is proved (Theorem~\mref{thm:dul}). It is further shown that $\tot{\spp}_\Omega$ can be obtained from the Manin white square product of $\spp$ with the totally compatible associative operad $\tot{\ass}_\Omega$ (Corollary~\mref{cor:totalwhite}).

See~\mcite{FMZ} for an operadic study of a related structure, called family algebraic structures.

\smallskip

{\bf Notation.}
Throughout this paper, we will work over a fixed field $\bfk$ of characteristic zero. It is the base field for all vector spaces, algebras,  tensor products, as well as linear maps.
\vspace{-.1cm}
\section{Linear compatibility and the Manin black square product}
\mlabel{sec:lincom}
In this section, we introduce the notion of \ubqcos and then study for such operads the first of our compatibility conditions, that is, the linear compatibility condition.
\vspace{-.1cm}
\subsection{\Ubqcos}
\mlabel{ss:ubqco}
We give the notion and examples of \ubqcos and refer the reader to~\mcite{BD,LV,MSS} for further details on operads.
\begin{defn}
A \textbf{nonsymmetric (ns) operad} is an arity graded vector space
$\mathscr{P}= \{\mathscr{P}_0, \mathscr{P}_1, \cdots\}$ equipped with an element $\id\in\mscr{P}_1$ and
{\bf composition maps}:
$$
\gamma:=\gamma_{n_1,\ldots,n_k}^k:\mscr{P}_{k}\otimes\mscr{P}_{n_1}\otimes\cdots\otimes\mscr{P}_{n_k}
\longrightarrow\mscr{P}_{n_1+\cdots+n_k}, \quad (\mu;\nu_1,\ldots,\nu_k) \mapsto \gamma(\mu;\nu_1,\ldots,\nu_k)
$$
which satisfy the following properties:
\begin{enumerate}
  \item (associativity)
\begin{multline*}
 \gamma(\mu;\gamma(\nu_1;\omega_{1,1},\ldots,\omega_{1,\ell_1}),\ldots,\gamma(\nu_k;\omega_{k,1},\ldots,\omega_{k,\ell_k}))\\
=\gamma(\gamma(\mu;\nu_1,\ldots,\nu_k);\omega_{1,1},\ldots,\omega_{1,\ell_1},\omega_{2,1},\ldots,\omega_{k,1},\ldots,\omega_{k,\ell_k}).\mlabel{it:asso}
\end{multline*}
\item (unitality)
$\gamma(\id;\mu)=\mu,\quad \gamma(\mu;\id,\ldots,\id)=\mu.$\mlabel{it:unit}
\end{enumerate}
\mlabel{defn:nsop}
\end{defn}
An ns operad $\mathscr{P}$ is called {\bf reduced} if $\mathscr{P}_0 = 0$.
All operads are assumed to be reduced in this paper. A morphism of ns operads is a morphism of the arity graded vector spaces that is compatible with the compositions.

Recall that, for any vector space $V$, the pair $\text{End}_V:= \big( \{{\rm Hom}(V^{\otimes n},V)\}_{n\geq0},\gamma\big)$ is an ns operad,
with $\gamma(f;g_1,\ldots,g_k)$ the usual composition of multivariate functions.
\begin{defn}
	Let $\mscr{P}$ be an ns operad.
\begin{enumerate}
  \item A \textbf{$\mscr{P}$-algebra} is a vector space $V$
	with a morphism of ns operads $\mscr{P} \overset{\rho}{\to} \text{End}_V$. We say that the operad $\mscr{P}$ {\bf encodes} the $\mscr{P}$-algebras.
  \item Let $(V, \rho_V)$ and $(W, \rho_W)$ be $\mscr{P}$-algebras. A {\bf morphism of $\mscr{P}$-algebras} is a linear map $f : V\to W$
  such that
  $$f(\rho_V(\mu)(v_1, \ldots,v_n))=\rho_W(\mu)(f(v_1), \ldots,f(v_n)),
  \text{ for all } n\geq0, \mu\in\spp_n, v_i\in V.$$
\end{enumerate}
\mlabel{def:palg}
\end{defn}
Let $E=\{E_0,E_1,\ldots\}$ be an arity graded vector space.
The free ns operad $\mscr{T}(E)$ on $E$ can be constructed as follows.
Let $\mscr{T}(E)_n$ be the vector space spanned by all decorated planar rooted trees with $n$ leaves whose each internal (non-leave) vertex $v$ is decorated by an element of $E_{|\text{in}(v)|}$, where $\text{in}(v)$ is the set of inputs of the vertex $v$ in
the planar rooted tree.
Consider the arity graded vector space
$$\mscr{T}(E):=\{\mscr{T}(E)_0,\mscr{T}(E)_1,\ldots\}.$$
Define a composition product $\gamma(t;t_1,\ldots,t_n)$ by grafting the root of $t_i\in \mscr{T}(E)_{n_i}$ to the $i$-th leaf of $t\in \mscr{T}(E)_k$ for $1\leq i\leq k$. This composition product $\gamma$ makes $\mscr{T}(E)$ into an ns operad.
Let $i:E\rightarrow \mscr{T}(E)$ denote the embedding map identifying the operation $\mu\in E(n)$ with the $n$-th corolla decorated by $\mu$.
\begin{lemma}\mcite{LV}
Let $E=\{E_0,E_1,\ldots\}$ be an arity graded vector space.
Then $(\mscr{T}(E),\gamma)$ together with the natural embedding $i: E\to \mscr{T}(E)$
is the free ns operad on $E$.
\end{lemma}
Any ns operad $\mscr{P}$ can be presented as the quotient of a free ns operad modulo an operadic ideal:
$$\mscr{P} =\mscr{P}(E, R): = \mscr{T}(E)/\langle R\rangle,$$
where $E$ and $R$ are called the {\bf generators} and {\bf relations} of $\mscr{P}$, respectively. $\mscr{P}$ is called {\bf finitely generated} if $E_n$ is finite dimensional for all $n\geq 0$.

For any element $t$ of $\mscr{T}(E)$, its {\bf weight} is defined to the number of internal vertices of $t$
as a decorated planar rooted tree.
For example, for the following planar decorated rooted trees in $\mscr{T}(E)$,
  $$\treey{\cdlr{o}\node at (0,0.2) {$\mu$};}, \quad
  \treey{\cdlr[0.8]{o} \cdl{ol}\cdr{or}\zhds{o/b}
\node at (0.2,-0.2) {$d$};\node at (0,0.2) {$\mu$};},\quad
\treey{\cdlr{ol}\node at (0,0.2) {$\mu$};\node at  (ol) [above] {$\nu$};}, \quad
\treey{\cdlr[0.8]{o}\cdl{ol}\cdr{or}
\node at (ol) [scale=0.6]{$\bullet$};\node at (or) [scale=0.6]{$\bullet$};\node at (-0.4,0.1) {$P$};\node at (0.4,0.1) {$P$};\node at (0,0.2) {$\mu$};}\quad
,\treey{\cdlr[0.8]{o} \cdl{ol}\cdr{or}\zhds{o/b}
\node at (ol) [scale=0.6]{$\bullet$};\node at (-0.4,0.1) {$P$};\node at (0.2,-0.2) {$P$};\node at (0,0.2) {$\mu$};} \quad
,\treey{\cdlr[0.8]{o} \cdl{ol}\cdr{or}\zhds{o/b}
\node at (or) [scale=0.6]{$\bullet$};\node at (0.2,-0.2) {$P$};\node at (0.4,0.1) {$P$};\node at (0,0.2) {$\mu$};}
,$$
the first one has weight 1, the second and third ones have weight 2 and the last three ones have weight 3.
Denote by $\mscr{T}(E)^{(n)}$ the subset of elements of $\mscr{T}(E)$ of weight $n$. In particular,
$$\mscr{T}(E)^{(0)}=\bfk \id\,\text{ and }\, \mscr{T}(E)^{(1)}=E.$$
In this paper, we focus on unary binary \qc operads defined as follows.
\begin{defn}
Let $\mscr{P} = \mscr{T}(E)/\langle R\rangle $ be an ns operad.
\begin{enumerate}
\item We call $\mscr{P}$ {\bf unary binary} if
$E=\{0, E_1, E_2, 0,\ldots, 0 ,\ldots\}.$
If $E_1=\bfk\id$, then the operad is called {\bf binary}.
\item A relation in $R$ is said to be {\bf homogeneous} if it is in
$\mscr{T}(E)^{(k)}$ for some $k\geq 0$. In particular, it is called
{\bf quadratic} (resp. {\bf cubic}) if it is in
$\mscr{T}(E)^{(2)}$ (resp. $\mscr{T}(E)^{(3)}$).
We call $\mscr{P}$ {\bf \qc} if each relation in $R$ is either quadratic or cubic.
\end{enumerate}
\mlabel{defn:ubo}
\end{defn}
Note that a \qc operad can have quadratic relations and cubic relations at the same time, but cannot have a quadratic term and a cubic term in the same relation. For example, the relation for a Rota-Baxter algebra (see Example~\mref{ex:nsop}.\mref{ex:rba} and \mref{ex:Nij}) is \qc when the weight is zero, but is not \qc when the weight is nonzero.

For later applications, we give an explicit description of unary binary \qc ns operads.
For distinction, we will use suffix notion for arity of relations. More precisely, an $n$-ary relation will be called a relation in arity $n$.

A unary binary \qc nc operad $\mscr{P}$ can be presented by
\begin{equation}
\mscr{P}=\mscr{P}(E,R)=\mscr{T}(E)/\langle R\rangle ,
\mlabel{eq:ubqcop}
\end{equation}
for which
\begin{enumerate}
	\item
$E=\{\,0, E_1, E_2, 0,\ldots, 0 ,\ldots\,\},$
where
\begin{enumerate}
	\item
 $E_1$ is spanned by
\begin{equation}
	\biggr\{
\treeyy[scale=0.8]{\cdu{o} \node  at (0,0) [scale=0.6]{$\bullet$};\node at (0.3,0) {$P_1$};},
\treeyy[scale=0.8]{\cdu{o} \node  at (0,0) [scale=0.6]{$\bullet$};\node at (0.3,0) {$P_2$};},~...~,
\treeyy[scale=0.8]{\cdu{o} \node  at (0,0) [scale=0.6]{$\bullet$};\node at (0.3,0) {$P_t$};}
\biggr\},
\mlabel{eq:unaryop}
\end{equation}
\item $E_2$ is spanned by
\begin{equation}
	\left\{
\treey{\cdlr{o}\node  at (0,0.25) {$1$};},
\treey{\cdlr{o}\node  at (0,0.25) {$2$};},~... ~,
\treey{\cdlr{o}\node  at (0,0.25) {$s$};}
\right\}.
\mlabel{eq:binaryop}
\end{equation}
\end{enumerate}
\item
$R :=R_1\sqcup R_2 \sqcup R_3\sqcup R_4:= (R_{1,2}\sqcup R_{1,3})\sqcup (\Rbb \sqcup \Rbc )\sqcup (R_{3,2}\sqcup R_{3,3})\sqcup R_4$, where
\begin{enumerate}
	\item $R_{1,2}$ is the set of quadratic relations in arity one:
\begin{equation}
R_{1,2} := ~\biggr\{r_{1,2}^n(P_k,P_\ell):=\sum_{1\leq k,\ell\leq t}\alpha_{k,\ell}^{n}~~
\treeyy{\cdu{o} \node  at (0,-0.2) [scale=0.6]{$\bullet$};\node at (0.3,-0.2) {$P_\ell$};
\node  at (0,0.2) [scale=0.6]{$\bullet$};\node at (0.3,0.2) {$P_k$};}
 \quad\Big|\,{1\leq n\leq n_{1,2}}\biggr\},
\mlabel{eq:pr1}
\end{equation}
\item $R_{1,3}$ is the set of cubic relations in arity one:
\begin{equation}
R_{1,3} := ~\biggr\{r_{1,3}^n(P_k,P_\ell, P_m):=\sum_{1\leq k,\ell, m\leq t}\alpha_{k,\ell, m}^{n}~~
\treeyy{\cdu{o} \node  at (0,-0.3) [scale=0.6]{$\bullet$};\node at (0.3,-0.3) {$P_m$};
\node  at (0,0.3) [scale=0.6]{$\bullet$};\node at (0.3,0.3) {$P_k$};
\node  at (0,0) [scale=0.6]{$\bullet$};\node at (-0.3,0) {$P_\ell$};}
 \quad\Big|\,{1\leq n\leq n_{1,3}}\biggr\},
\mlabel{eq:pr13}
\end{equation}
\item $\Rbb$ is the set of quadratic relations in arity two:
\begin{equation}
\Rbb:= ~\biggr\{r_{2,2}^n(P_k,i):=\sum_{1\leq i\leq s, \atop 1\leq k\leq t}\kb{k}{i}{1} \treey{\cdlr[0.8]{o} \cdl{ol}\cdr{or}\zhds{o/b}
	\node at (0.25,-0.2) {$P_k$};\node at (0,0.3) {$i$};\node at (0,0.3) {$i$};}
+\kb{k}{i}{2}\treey{\cdlr[0.8]{o} \cdl{ol}\cdr{or}
	\node at (ol) [scale=0.6]{$\bullet$};\node at (-0.4,0.1) {$P_k$};\node at (0,0.3) {$i$};}
+\kb{k}{i}{3}\treey{\cdlr[0.8]{o} \cdl{ol}\cdr{or}
	\node at (or) [scale=0.6]{$\bullet$};\node at (0.4,0.1) {$P_k$};\node at (0,0.3) {$i$};}
\quad\Big|\,{1\leq n\leq n_{2,2}}\biggr\},
\mlabel{eq:pr2-2}
\end{equation}
\item $\Rbc$ is the set of cubic relations in arity two:
\begin{align}
\Rbc
:=&~\biggr\{r_{2,3}^{n}(P_k,P_\ell,i):=\sum_{1\leq i\leq s, \atop 1\leq k,\ell\leq t}
\kc{k}{\ell}{i}{1}\treey{\cdlr[0.8]{o}\cdl{ol}\cdr{or}
\node at (ol) [scale=0.6]{$\bullet$};\node at (or) [scale=0.6]{$\bullet$};\node at (-0.4,0.1) {$P_{k}$};\node at (0.4,0.1) {$P_\ell$};\node at (0,0.3) {$i$};}
+\kc{k}{\ell}{i}{2}\treey{\cdlr[0.8]{o} \cdl{ol}\cdr{or}\zhds{o/b}
\node at (ol) [scale=0.6]{$\bullet$};\node at (-0.4,0.1) {$P_k$};\node at (0.25,-0.2) {$P_\ell$};\node at (0,0.3) {$i$};}
+\kc{k}{\ell}{i}{3}\treey{\cdlr[0.8]{o} \cdl{ol}\cdr{or}\zhds{o/b}
\node at (or) [scale=0.6]{$\bullet$};\node at (0.25,-0.2) {$P_k$};\node at (0.4,0.1) {$P_\ell$};\node at (0,0.3) {$i$};}
+\kc{k}{\ell}{i}{4}\treey{\cdlr[0.8]{o} \cdl{ol}\cdr{or}\node at (0,-0.3) [scale=0.6]{$\bullet$};
\node at (0,-0.1) [scale=0.6]{$\bullet$};\node at (0.25,-0.1) {$P_k$};\node at (0.25,-0.45) {$P_\ell$};\node at (0,0.3) {$i$};}\nonumber\\
&\qquad \qquad \qquad \qquad +\kc{k}{\ell}{i}{5}\treey{\cdlr[0.8]{o} \cdl{ol}\cdr{or}\node at (-0.2,0.2) [scale=0.6]{$\bullet$};
\node at (-0.4,0.4) [scale=0.6]{$\bullet$};\node at (-0.6,0.35) {$P_k$};\node at (-0.3,0) {$P_\ell$};\node at (0,0.3) {$i$};}
+\kc{k}{\ell}{i}{6}\treey{\cdlr[0.8]{o} \cdl{ol}\cdr{or}\node at (0.2,0.2) [scale=0.6]{$\bullet$};
\node at (0.4,0.4) [scale=0.6]{$\bullet$};\node at (0.6,0.35) {$P_k$};\node at (0.3,0) {$P_\ell$};\node at (0,0.3) {$i$};}\quad  \Big| \,{1\leq n\leq n_{2,3}}\biggr\},
\mlabel{eq:pr2-3}
\end{align}
\item $R_{3,2}$ is the set of quadratic relations in arity three:
\begin{equation}
R_{3,2}
:= ~\biggr\{r_{3,2}^n(i,j):=\sum_{1\leq i,j\leq s}\ka{i}{j}{1}
\treey{\cdlr{ol}
\foreach \i/\j in {ol/$i$,o/$j$} {\node[above] at (\i) {\j};}}
+\ka{i}{j}{2}\treey{\cdlr{or}
\foreach \i/\j in {or/$j$,o/$i$} {\node[above] at (\i) {\j};}}
\quad \Big| \,{1\leq n\leq n_{3,2}}\biggr\},
\mlabel{eq:pr3}
\end{equation}
\item $R_{3,3}$ is the set of cubic relations in arity three:
\begin{equation}
\begin{split}
R_{3,3}:= ~\biggr\{&r_{3,3}^n(P_k, i,j):=\sum_{1\leq k\leq t \atop  1\leq i,j\leq s}
\kab{k}{i}{j}{1}\treey{\cdlr{ol}\foreach \i/\j in {ol/$i$,o/$j$} {\node[above] at (\i) {\j};}\node at (-0.65,0.3) {$P_k$};\node at (-0.5,0.5) [scale=0.6]{$\bullet$};}
+\kab{k}{i}{j}{2}\treey{\cdlr{ol}\node at (-0.45,0.2) {$i$};\node at (-0.2,-0.1) {$j$};\node at (0.05,0.4) {$P_k$};\node at (-0.15,0.5) [scale=0.6]{$\bullet$};}
+\kab{k}{i}{j}{3}\treey{\cdlr{ol}\foreach \i/\j in {ol/$i$,o/$j$} {\node[above] at (\i) {\j};}\node at (-0.25,0) {$P_k$};\node at (-0.2,0.2) [scale=0.6]{$\bullet$};}
+\kab{k}{i}{j}{4}\treey{\cdlr{ol}\foreach \i/\j in {ol/$i$,o/$j$} {\node[above] at (\i) {\j};}\node at (0.35,0) {$P_k$};\node at (0.2,0.2) [scale=0.6]{$\bullet$};}
+\kab{k}{i}{j}{5}\treey{\cdlr{ol}\foreach \i/\j in {ol/$i$,o/$j$} {\node[above] at (\i) {\j};}\node at (0.25,-0.2) {$P_k$};\node at (0,-0.2) [scale=0.6]{$\bullet$};}\\
&+\kab{k}{i}{j}{6}\treey{\cdlr{or}\node at (0.45,0.2) {$i$};\node at (-0.2,-0.1) {$j$};\node at (-0.05,0.4) {$P_k$};\node at (0.15,0.5) [scale=0.6]{$\bullet$};}
+\kab{k}{i}{j}{7}\treey{\cdlr{or}\foreach \i/\j in {or/$j$,o/$i$} {\node[above] at (\i) {\j};}\node at (0.65,0.3) {$P_k$};\node at (0.5,0.5) [scale=0.6]{$\bullet$};}
+\kab{k}{i}{j}{8}\treey{\cdlr{or}\foreach \i/\j in {or/$j$,o/$i$} {\node[above] at (\i) {\j};}\node at (-0.3,0) {$P_k$};\node at (-0.2,0.2) [scale=0.6]{$\bullet$};}
+\kab{k}{i}{j}{9}\treey{\cdlr{or}\foreach \i/\j in {or/$j$,o/$i$} {\node[above] at (\i) {\j};}\node at (0.3,0) {$P_k$};\node at (0.2,0.2) [scale=0.6]{$\bullet$};}
+\kab{k}{i}{j}{10}\treey{\cdlr{or}\foreach \i/\j in {or/$j$,o/$i$} {\node[above] at (\i) {\j};}\node at (0.25,-0.2) {$P_k$};\node at (0,-0.2) [scale=0.6]{$\bullet$};}
\quad \Big| \,{1\leq n\leq n_{3,3}}\biggr\},
\end{split}
\mlabel{eq:pr33}
\end{equation}
\item $R_{4}$ is the set of cubic relations in arity four:
\begin{equation}
R_{4}:= ~\biggr\{r_{4}^n(i,j,p):=\sum_{1\leq i,j,p\leq s}
\kd{i}{j}{p}{1}\treey{\cdlr{ol}\cdlr{oll}\foreach \i/\j in {oll/$i$,ol/$j$,o/$p$}  {\node[above] at (\i) {\j};}}
+\kd{i}{j}{p}{2}\treey{\cdlr{ol}\cdlr{olr}\foreach \i/\j in {olr/$j$,ol/$i$,o/$p$}  {\node[above] at (\i) {\j};}}
+\kd{i}{j}{p}{3}\treey{\cdlr[1.2]o\cdlr{ol}\cdlr{or}\foreach \i/\j in {ol/$i$,or/$p$,o/$j$}  {\node[above] at (\i) {\j};}}
+\kd{i}{j}{p}{4}\treey{\cdlr{or}\cdlr{orl}\foreach \i/\j in {orl/$j$,or/$p$,o/$i$}  {\node[above] at (\i) {\j};}}
+\kd{i}{j}{p}{5}\treey{\cdlr{or}\cdlr{orr}\foreach \i/\j in {orr/$p$,or/$j$,o/$i$} {\node[above] at  (\i) {\j};}}
\  \Big| \,{1\leq n\leq n_{4}}\biggr\}.
\mlabel{eq:pr4}
\end{equation}
\end{enumerate}
\end{enumerate}
Here the Greek letters are coefficients in $\bfk$, and $n_{1,2}, n_{1,3}, n_{2,2}, n_{2,3}, n_{3,2}, n_{3,3}$ and $n_4$ are nonnegative integers, with the convention that when any of the integers is zero, then the corresponding set is empty.
Note that if $\spp$ is a quadratic operad, then the above relations are reduced to
\begin{equation*}
R =R_1\sqcup R_2 \sqcup R_3= R_{1,2}\sqcup \Rbb \sqcup R_{3,2}.
\end{equation*}

We give the following examples to fix notations for later applications and to demonstrate the wide range of structures covered by the notion of \ubqcos.
\begin{exam}\mlabel{ex:nsop}
\begin{enumerate}
\item The \name{associative algebra} is encoded by the associative operad $\as$, which is binary quadratic and defined by
      $$E=E_2=\bfk\treey{\cdlr{o}}\,\text{ and }\, R=R_{3,2}=\left\{\treey{\cdlr{ol}}-\treey{\cdlr{or}}\right\}.$$
\mlabel{exam:0}
\item The \name{differential associative algebra}~\mcite{Ko} is encoded by the differential associative operad $\dera$~\mcite{Lod10}, defined by the unary and binary generators
	$$E_1 = \bfk\treeyy[scale=0.8]{\cdu{o} \node  at (0,0) [scale=0.6]{$\bullet$};\node at (0.3,0) {$d$};}, \quad
	E_2 = \bfk\treey{\cdlr{o}}$$
	and quadratic relations
	$$~R_{2}=R_{2,2}=\left\{
		\treey{\cdlr[0.8]{o} \cdl{ol}\cdr{or}\zhds{o/b}
			\node at (0.2,-0.2) {$d$};}
		-\treey{\cdlr[0.8]{o} \cdl{ol}\cdr{or}
			\node at (ol) [scale=0.6]{$\bullet$};\node at (-0.2,0.5) {$d$};}
		-\treey{\cdlr[0.8]{o} \cdl{ol}\cdr{or}
			\node at (or) [scale=0.6]{$\bullet$};\node at (0.2,0.5) {$d$};}\right\},\quad
		~R_{3}=R_{3,2}=\left\{\treey{\cdlr{ol}}-\treey{\cdlr{or}}\right\}
		$$
	corresponding to the Leibniz rule and the associativity law. \mlabel{ex:diff}
\item The \name{Rota-Baxter associative algebra} (of weight zero)~\mcite{Bax,Gub} is encoded by the Rota-Baxter associative operad $\rba$, defined by the unary binary generators
	$$E_1 = \bfk\treeyy[scale=0.8]{\cdu{o} \node  at (0,0) [scale=0.6]{$\bullet$};\node at (0.3,0) {$P$};}, \quad
	E_2=\bfk\treey{\cdlr{o}}$$
	and the quadratic and cubic relations
	$$~R_{2}=R_{2,3}=\left\{\treey{\cdlr[0.8]{o}\cdl{ol}\cdr{or}
			\node at (ol) [scale=0.6]{$\bullet$};\node at (or) [scale=0.6]{$\bullet$};\node at (-0.2,0.5) {$P$};\node at (0.2,0.5) {$P$};}
		-\treey{\cdlr[0.8]{o} \cdl{ol}\cdr{or}\zhds{o/b}
			\node at (ol) [scale=0.6]{$\bullet$};\node at (-0.2,0.5) {$P$};\node at (0.2,-0.2) {$P$};}
		-\treey{\cdlr[0.8]{o} \cdl{ol}\cdr{or}\zhds{o/b}
			\node at (or) [scale=0.6]{$\bullet$};\node at (0.2,-0.2) {$P$};\node at (0.2,0.5) {$P$};}\right\},\quad
		~R_{3}=R_{3,2}=\left\{\treey{\cdlr{ol}}-\treey{\cdlr{or}}\right\}.
		$$
\mlabel{ex:rba}
\item The Rota-Baxter algebra with nonzero weight $\lambda$ does not give a \ubqco since the Rota-Baxter relation
     $$ P(x)P(y)-P(P(x)y)-P(xP(y))-\lambda P(xy)=0$$
      contains both quadratic and cubic terms. As its homogeneous analog, the  \name{Nijenhuis associative algebra} is encoded by the \ubqco $\na$, with generators
	$$E_1 =\bfk\treeyy[scale=0.8]{\cdu{o} \node  at (0,0) [scale=0.6]{$\bullet$};\node at (0.3,0) {$P$};},
		\quad E_2 = \bfk\treey{\cdlr{o}}$$
	and relations
	$$~R_{2}=R_{2,3}=\left\{\treey{\cdlr[0.8]{o}\cdl{ol}\cdr{or}
			\node at (ol) [scale=0.6]{$\bullet$};\node at (or) [scale=0.6]{$\bullet$};\node at (-0.2,0.5) {$P$};\node at (0.2,0.5) {$P$};}
		-\treey{\cdlr[0.8]{o} \cdl{ol}\cdr{or}\zhds{o/b}
			\node at (ol) [scale=0.6]{$\bullet$};\node at (-0.2,0.5) {$P$};\node at (0.2,-0.2) {$P$};}
		-\treey{\cdlr[0.8]{o} \cdl{ol}\cdr{or}\zhds{o/b}
			\node at (or) [scale=0.6]{$\bullet$};\node at (0.2,-0.2) {$P$};\node at (0.2,0.5) {$P$};}
		+\treey{\cdlr[0.8]{o} \cdl{ol}\cdr{or}\node at (0,-0.1) [scale=0.6]{$\bullet$};\node at (0,-0.3) [scale=0.6]{$\bullet$};\node at (0.2,-0.1) {$P$};\node at (0.2,-0.3) {$P$};}\right\},\quad ~R_{3}=R_{3,2}=\left\{\treey{\cdlr{ol}}-\treey{\cdlr{or}}\right\}.$$
\mlabel{ex:Nij}
\item Let $\Delta$ be a nonempty set. A \name{$\Delta$-differential algebra}~\mcite{Ko} is an algebra with multiple differential operators $\delta\in \Delta$ that commute with each other.
	This algebraic structure is encoded by the unary binary quadratic ns operad, with generators
		$$E_1 = \bfk\left\{\treeyy[scale=0.8]{\cdu{o} \node  at (0,0) [scale=0.6]{$\bullet$};\node at (0.3,0) {$\delta$};}\,\big|\,\delta\in\Delta\right\}, \quad
		E_2 = \bfk\treey{\cdlr{o}}$$
	and relations
\begin{gather*}
R_{1}=R_{1,2}=\left\{\treeyy{\cdu{o} \node  at (0,-0.2) [scale=0.6]{$\bullet$};\node at (0.3,-0.2) {$\delta_1$};
			\node  at (0,0.2) [scale=0.6]{$\bullet$};\node at (0.3,0.2) {$\delta_2$};}-\treeyy{\cdu{o} \node  at (0,-0.2) [scale=0.6]{$\bullet$};\node at (0.3,-0.2) {$\delta_2$};
			\node  at (0,0.2) [scale=0.6]{$\bullet$};\node at (0.3,0.2) {$\delta_1$};}\,\big|\,\delta_1,\delta_2\in\Delta\right\},\quad
		~R_{2}=R_{2,2}=\left\{
		\treey{\cdlr[0.8]{o} \cdl{ol}\cdr{or}\zhds{o/b}
			\node at (0.2,-0.2) {$\delta$};}
		-\treey{\cdlr[0.8]{o} \cdl{ol}\cdr{or}
			\node at (ol) [scale=0.6]{$\bullet$};\node at (-0.2,0.5) {$\delta$};}
		-\treey{\cdlr[0.8]{o} \cdl{ol}\cdr{or}
			\node at (or) [scale=0.6]{$\bullet$};\node at (0.2,0.5) {$\delta$};}\,\Big|\,\delta\in\Delta\right\},\quad\\
		~R_{3}=R_{3,2}=\left\{\treey{\cdlr{ol}}-\treey{\cdlr{or}}\right\}.
\end{gather*}
\mlabel{item:ddas}
\item A \name{Hom-associative algebra}~\mcite{MS} is encoded by the Hom-associative operad $\homop$ defined by the unary and binary generators
	$$E_1 =\bfk\treeyy[scale=0.8]{\cdu{o} \node  at (0,0) [scale=0.6]{$\bullet$};\node at (0.3,0) {$\alpha$};},
		\quad E_2 = \bfk\treey{\cdlr{o}}$$
and cubic relation
$$R_{3}=R_{3,3}=\left\{\treey{\cdlr{ol}\node at (0.1,0.35) {$\alpha$};\node at (0.2,0.2) [scale=0.6]{$\bullet$};}-\treey{\cdlr{or}\node at (-0.1,0.35) {$\alpha$};\node at (-0.2,0.2) [scale=0.6]{$\bullet$};}\right\}.$$
\mlabel{it:homas}
\item A \name{cubic associative algebra}~\mcite{BD} is encoded by the cubic associative operad $\cubas$ define by the binary generator
$$E=E_2 = \bfk\treey{\cdlr{o}}$$
and the cubic relations in arity four
\begin{equation*}
R=R_4=\Biggr\{ \treey{\cdlr{ol}\cdlr{oll}}-\treey{\cdlr{ol}\cdlr{olr}},\quad\treey{\cdlr{ol}\cdlr{oll}}-\treey{\cdlr[1.2]o\cdlr{ol}\cdlr{or}},\quad
\treey{\cdlr{ol}\cdlr{oll}}-\treey{\cdlr{or}\cdlr{orl}},\quad\treey{\cdlr{ol}\cdlr{oll}}-\treey{\cdlr{or}\cdlr{orr}}\Biggr\}.
\end{equation*}
\mlabel{it:cubas}
	\end{enumerate}
\end{exam}

\subsection{Linearly compatible operads}\mlabel{sec:lincom1}

The linear compatibility of a binary quadratic operad with two duplicated copies of  operations has been studied in~\mcite{St}.
\begin{defn}\cite[Definition~A]{St}
	Let $\mscr{P}$ be a binary quadratic operad and $V$ a vector space. Let
	$$A = (V, \mu_1, \ldots, \mu_k)\,\text{ and }\, B = (V, \nu_1, \ldots, \nu_k)$$ be two $\mscr{P}$-algebras such that $\mu_i$ and $\nu_i$ are the corresponding
	binary operations for $1\leq i\leq k$. Denote $\eta_i:= \alpha \mu_i + \beta \nu_i$ for some $\alpha, \beta\in \bfk$.
	The pair $A, B$ are called {\bf linearly compatible} if $C = (V, \eta_1, \ldots, \eta_k)$ is a $\mscr{P}$-algebra for any choice of
	$\alpha$ and $\beta$. This is equivalent to requiring that $C$ is a $\mscr{P}$-algebra for $\alpha = \beta =1$.
	\mlabel{de:lincomp}
\end{defn}
We will study the more general case of a unary binary operad with \qc relations, with any number of replicated copies of the unary and binary operations.
To motivate the general notion, we first discuss an example.
\begin{exam}
Let $V$ be a vector space. Let $A = (V, \multa, \opera)$ and $B = (V, \multb, \operb)$ be Rota-Baxter associative algebras of weight zero on the same underlying vector space. Define
$$\ast:=\alpha\multa+\beta\multb\,\text{ and }\, R:=\alpha \opera+\beta \operb \,\text{ for }\, \alpha,\beta\in\bfk.$$
The two Rota-Baxter algebras are called {\bf linearly compatible} if the triple $(V,\ast, R)$ is still a Rota-Baxter associative algebra of weight zero.

Imposing the associativity for $\ast$ and the Rota-Baxter relation for $R$, we find that $(V, \ast, R)$ is a Rota-Baxter associative algebra for all choices of $\alpha$ and $\beta$ if and only if
\begin{align*}
\prl{\multa}{\multb}+\prl{\multb}{\multa}
=\prr{\multa}{\multb}+\prr{\multb}{\multa},
\end{align*}
and
\begin{equation*}
\begin{split}
&\ra{\opera}{\multb}{\opera} +\ra{\opera}{\multa}{\operb}+\ra{\operb}{\multa}{\opera}\\
=&\rb{\opera}{\multb}{\opera}+\rb{\opera}{\multa}{\operb}\\
&+\rb{\operb}{\multa}{\opera},\\
&\ra{\opera}{\multb}{\operb}+\ra{\operb}{\multb}{\opera}+\ra{\operb}{\multa}{\operb}\\
=&\rb{\opera}{\multb}{\operb}+\rb{\operb}{\multb}{\opera} \\
&+\rb{\operb}{\multa}{\operb}.
\end{split}
\end{equation*}
These should be the relations for the operad encoding linearly compatible Rota-Baxter algebras with two replicated copies.
\mlabel{exam:rbcom}
\end{exam}

Generalizing this to algebras on an ns operad, we give
\begin{defn}
Let $\Omega$ be a nonempty set. Let $\mscr{P} = \mscr{T}(E)/\langle R\rangle $ be an operad with generators $\mu_1,\ldots,\mu_k$ and such that each relation in $R$ is homogeneous. For each $\omega$ in a nonempty set $\Omega$, let $\mu_{\omega,1}, \ldots, \mu_{\omega,k}$ be a replicate of $\mu_1,\ldots,\mu_k$.
Let $V$ be a vector space such that, for each $\omega\in \Omega$, the tuple
$$A_\omega \coloneqq (V, \mu_{\omega,1}, \ldots, \mu_{\omega,k}) $$
is a $\mscr{P}$-algebra. Then $V$ is called an {\bf ($\Omega$-)linearly compatible $\mscr{P}$-algebra} if, for any $\bfc\coloneqq(c_\omega)_{\omega\in \Omega}$ with $c_\omega\in \bfk$ and the linear combination $\nu_i:= \nu_{\bfc,i}:=\sum_{\omega\in\Omega} c_{\omega} \mu_{\omega,i}, i=1,\ldots,k$, the tuple $$A:=A_\bfc: = (V, \nu_1, \ldots, \nu_k)$$
is still a $\mscr{P}$-algebra.
\mlabel{defn:lcalg}
\end{defn}

Since $\mu_{\omega, i}, \omega\in \Omega,$ are spacial cases of $\nu_i, i=1,\ldots,k$, we have

\begin{prop}
Let $\Omega$ be a nonempty set and let $\mscr{P}$ be an operad with generators $\mu_1,\ldots,\mu_k$ and such that each relation in $R$ is homogeneous.
A vector space $V$ is an ($\Omega$-)linearly compatible $\mscr{P}$-algebra if and only if $(V,\nu_1,\ldots,\nu_k)$ is a $\mscr{P}$-algebra for any linear combination $\nu_i:= \sum_{\omega\in\Omega} c_{\omega} \mu_{\omega,i}, c_{\omega}\in \bfk, \omega\in \Omega$.
\mlabel{defn:lcalg2}
\end{prop}

Now we determine the operad that encodes linearly compatible $\mscr{P}$-algebras for any unary binary \qc operad $\mscr{P}$.
Let $\mscr{P}=\mscr{T}(E)/\langle R\rangle $ be such an operad, with its operations and relations shown in~\meqref{eq:unaryop} -- \meqref{eq:pr4}.
Let $\Omega$ be a nonempty set.  We consider a family of ns operads
\begin{equation*}
	\mscr{P}_\omega=\mscr{T}(E_\omega)/ \langle R_\omega\rangle, \,\text{ for }\,\omega\in\Omega,
\end{equation*}
that are just copies of $\mscr{P}$.
As for $\spp$, we describe $\mscr{P}_\omega$ by the generators
\begin{equation}
E_{\omega,1}:= ~\bfk\biggr\{
\treeyy[scale=0.8]{\cdu{o} \node  at (0,0)[scale=0.6] {$\bullet$};\node at (0.35,0) {$P_{\omega,1}$};},
\treeyy[scale=0.8]{\cdu{o} \node  at (0,0) [scale=0.6]{$\bullet$};\node at (0.35,0) {$P_{\omega,2}$};},~...~,
\treeyy[scale=0.8]{\cdu{o} \node  at (0,0) [scale=0.6]{$\bullet$};\node at (0.35,0) {$P_{\omega,t}$};}
\biggr\}~~,
~~
E_{\omega,2}:= ~\bfk\biggr\{
\treey{\cdlr{o}\node  at (0,0.35) {$1_\omega$}; },
\treey{\cdlr{o}\node  at (0,0.35) {$2_\omega$};},~... ~,
\treey{\cdlr{o}\node  at (0,0.35) {$s_\omega$};}
\biggr\}\,,
E_{\omega,n}:=0,\, n\neq1,2, \omega\in\Omega.
\mlabel{eq:opcirc}
\end{equation}
The relations
\begin{equation}
R_\omega:=R_{\omega,1,2}\sqcup R_{\omega,1,3}\sqcup \Rbbw\sqcup\Rbcw \sqcup R_{\omega,3,2}\sqcup R_{\omega,3,3}\sqcup R_{\omega,4},\,\text{ for }\, \omega\in\Omega\mlabel{eq:recb}
\end{equation}
are given by the same coefficients as the corresponding ones of $R$.
For example,
$$R_{\omega,1,2}= ~\biggr\{r_{1,2}^n(P_{\omega,k},P_{\omega,\ell}):=\sum_{1\leq k,\ell\leq t}\alpha_{k,\ell}^{n}~~
\treeyy{\cdu{o} \node  at (0,-0.2)[scale=0.6] {$\bullet$};\node at (0.35,-0.2) {$P_{\omega,\ell}$};
\node  at (0,0.2)[scale=0.6] {$\bullet$};\node at (0.35,0.2) {$P_{\omega,k}$};}
\quad \Big| ~{1\leq n\leq n_{1,2}}\biggr\}.$$
Define the arity graded space
$$\bigoplus\limits_{\omega\in\Omega}E_\omega :=~\biggr\{0,\bigoplus\limits_{\omega\in\Omega}E_{\omega,1},\bigoplus\limits_{\omega\in\Omega}E_{\omega,2},0\ldots\biggr\}.$$
For each $\omega\in\Omega,$ by embedding $\spp_\omega$ into $\stt\big( \bigoplus\limits_{\omega\in\Omega}E_\omega\big)/\big\langle  \bigcup\limits_{\omega\in\Omega} R _\omega\big\rangle $,
a $\stt\big( \bigoplus\limits_{\omega\in\Omega}E_\omega\big)/\langle  \bigcup\limits_{\omega\in\Omega} R _\omega\rangle $-algebra is a family of $\mscr{P}$-algebras which are
not necessarily compatible with one another in any way.

To describe the compatible conditions among the copies $\mscr{P}_\omega$, we build on the notations in Eqs.~\meqref{eq:pr1} -- \meqref{eq:pr4} and define, for $\mu, \nu,\omega\in\Omega$,
\begin{align}
 R_{1,2}^{\mu,\nu}:=&~\big\{r_{1,2}^n(P_{\mu,k},P_{\nu,\ell})\, \big| \,{1\leq n\leq n_{1,2}}\big\},\mlabel{eq:pr2b}\\
 R_{1,3}^{\mu,\nu,\omega} :=& ~\big\{r_{1,3}^n(P_{\mu,k},P_{\nu,\ell}, P_{\omega,m})\,\big|\,{1\leq n\leq n_{1,3}}\big\},\mlabel{eq:pr1c}\\
 R_{2,2}^{\mu,\nu}:=&~\big\{r_{2,2}^n(P_{\mu,k},i_\nu)\, \big| \,{1\leq n\leq n_{2,2}}\big\},\mlabel{eq:pr2-2b}\\
 R_{2,3}^{\mu,\nu,\omega}:=&~\big\{r_{2,3}^n(P_{\mu,k},P_{\nu,\ell},i_\omega)\, \big| \,{1\leq n\leq n_{2,3}}\big\},\mlabel{eq:pr2-3b}\\
 R_{3,2}^{\mu,\nu}:=&~\big\{r_{3,2}^n(i_\mu,j_\nu)\,\big| \,{1\leq n\leq n_{3,2}}\big\},\mlabel{eq:pr3b}\\
 R_{3,3}^{\mu,\nu,\omega}:=&~\big\{r_{3,3}^n(P_{\mu,k}, i_\nu,j_\omega)\, \big| \,{1\leq n\leq n_{3,3}}\big\},\mlabel{eq:pr3c}\\
 R_{4}^{\mu,\nu,\omega}:=&~\big\{r_{4}^n(i_\mu,j_\nu,p_\omega)\, \big| \,{1\leq n\leq n_{4}}\big\}.\mlabel{eq:pr4c}
\end{align}
For example,
\begin{align*}
R_{1,2}^{\mu,\nu}
=& ~\biggr\{r_{1,2}^n(P_{\mu,k},P_{\nu,\ell}):=\sum_{1\leq k,\ell\leq t}\alpha_{k,\ell}^{n}~~
\treeyy{\cdu{o} \node  at (0,-0.2)[scale=0.6] {$\bullet$};\node at (0.35,-0.2) {$P_{\nu,\ell}$};
\node  at (0,0.2)[scale=0.6] {$\bullet$};\node at (0.35,0.2) {$P_{\mu,k}$};}
\quad \Big| \,{1\leq n\leq n_{1,2}}\biggr\},\\
R_{2,2}^{\mu,\nu}
=& ~\biggr\{r_{2,2}^n(P_{\mu,k},i_\nu):=\sum_{1\leq i\leq s, \atop 1\leq k\leq t}
\kb{k}{i}{1}\treey{\cdlr[0.8]{o} \cdl{ol}\cdr{or}\zhds{o/b}
			\node at (0.35,-0.2) {$P_{\mu,k}$};\node at (0,0.3) {$i_\nu$};\node at (0,0.3) {$i_\nu$};}
+\kb{k}{i}{2}\treey{\cdlr[0.8]{o} \cdl{ol}\cdr{or}
			\node at (ol) [scale=0.6]{$\bullet$};\node at (-0.5,0.1) {$P_{\mu,k}$};\node at (0,0.3) {$i_\nu$};}
+\kb{k}{i}{3}\treey{\cdlr[0.8]{o} \cdl{ol}\cdr{or}
			\node at (or) [scale=0.6]{$\bullet$};\node at (0.5,0.1) {$P_{\mu,k}$};\node at (0,0.3) {$i_\nu$};}
		\quad \Big| \,{1\leq n\leq n_{2,2}}\biggr\}.
\end{align*}
Note that, comparing the above notions with Eq.~\meqref{eq:recb}, we have
\begin{equation}
R_{i,2}^{\omega,\omega}=R_{\omega,i,2}, \,R_{i,3}^{\omega,\omega,\omega}=R_{\omega,i,3}\,\text{ and }\, R_4^{\omega,\omega,\omega}=R_{\omega,4},\,\text{ for }\, i={1,2,3} \,\text{ and } \omega\in\Omega.
\mlabel{eq:spr123}
\end{equation}

We will use the following general notation for our discussion.
\begin{notation}
Let $X: = \{x_i\,|\, i\in I\}$ and $Y: = \{y_j\,|\,j\in I\}$ be two sets of elements in a vector space $V$, parameterized by the same set $I$.
Denote
$X\add Y:= \{x_i+ y_i\,|\, i\in I\}.$
\mlabel{nn:add}
\end{notation}
In particular, when $I=[n]:=\{1,\ldots,n\}$, then the above $X$ and $Y$ can be regarded as vectors $X=(x_1,\ldots,x_n)$ and $Y=(y_1,\dots,y_n)$. Then $X\add Y$ is simply the addition of the two vectors.

Then together with the notations in Eqs.~\meqref{eq:pr2b} -- \meqref{eq:pr4c}, we have
\begin{align*}
R^{\mu,\nu}_{1,2}\add R^{\omega,\tau}_{1,2}&=  \Big\{r_{1,2}^n(P_{\mu,k},P_{\nu,\ell})+r_{1,2}^n(P_{\omega,k},P_{\tau,\ell})\, \big| \,{1\leq n\leq n_{1,2}}\Big\}\\
&=  \Big\{r_{1,2}^n(P_{\mu,k}+P_{\omega,k},P_{\nu,\ell}+P_{\tau,\ell})\,\big|\,{1\leq n\leq n_{1,2}}\Big\},\\
R^{\mu,\nu,\omega}_{1,3}\add R^{\lambda,\tau,\rho}_{1,3}&=  \Big\{r_{1,3}^n(P_{\mu,k},P_{\nu,\ell},P_{\omega,m})+r_{1,3}^n(P_{\lambda,k},P_{\tau,\ell},P_{\rho,m})\,\big|\,{1\leq n\leq n_{1,3}}\Big\}\\
&=  \Big\{r_{1,3}^n(P_{\mu,k}+P_{\lambda,k},P_{\nu,\ell}+P_{\tau,\ell},P_{\omega,m}+P_{\rho,m})\,\big|\,{1\leq n\leq n_{1,3}}\Big\}, \\
R^{\mu,\nu}_{2,2}\add R^{\omega,\tau}_{2,2}&=  \Big\{r_{2,2}^n(P_{\mu,k},i_\nu)+r_{2,2}^n(P_{\omega,k},i_\tau)\,\big|\,{1\leq n\leq n_{2,2}}\Big\}
=  \Big\{r_{2,2}^n(P_{\mu,k}+P_{\omega,k},i_\nu+i_\tau)\,\big|\,{1\leq n\leq n_{2,2}}\Big\}, \\
R^{\mu,\nu,\omega}_{2,3}\add R^{\lambda,\tau,\rho}_{2,3}&=  \Big\{r_{2,3}^n(P_{\mu,k},P_{\nu,\ell},i_\omega)+r_{2,3}^n(P_{\lambda,k},P_{\tau,\ell},i_\rho)\,\big|\,{1\leq n\leq n_{2,3}}\Big\}\\
&=  \Big\{r_{2,3}^n(P_{\mu,k}+P_{\lambda,k},P_{\nu,\ell}+P_{\tau,\ell},i_\omega+i_\rho)\,\big|\,{1\leq n\leq n_{2,3}}\Big\}, \\
R^{\mu,\nu}_{3,2}\add R^{\omega,\tau}_{3,2}&=  \Big\{r_{3,2}^n(i_\mu,j_\nu)+r_{3,2}^n(i_\omega,j_\tau) \,\big|\,{1\leq n\leq n_{3,2}}\Big\}
=  \Big\{r_{3,2}^n(i_\mu+i_\omega,j_\nu+j_\tau)
\,\big|\,{1\leq n\leq n_{3,2}}\Big\},\\
R_{3,3}^{\mu,\nu,\omega}\add R_{3,3}^{\lambda,\tau,\rho}&= \Big\{r_{3,3}^n(P_{\mu,k}, i_\mu,j_\omega)+r_{3,3}^n(P_{\lambda,k}, i_\tau,j_\rho)\, \big| \,{1\leq n\leq n_{3,3}}\Big\}\\
&= \Big\{r_{3,3}^n(P_{\mu,k}+P_{\lambda,k}, i_\mu+i_\tau,j_\omega+j_\rho)\, \big| \,{1\leq n\leq n_{3,3}}\Big\},\\
R_{4}^{\mu,\nu,\omega}\add R_{4}^{\lambda,\tau,\rho}&= \Big\{r_{4}^n(i_\mu,j_\nu,p_\omega)+r_{4}^n(i_\lambda,j_\tau,p_\rho)\, \big| \,{1\leq n\leq n_{4}}\Big\}
= \Big\{r_{4}^n(i_\mu+i_\lambda,j_\nu+j_\tau,p_\omega+p_\rho)\, \big| \,{1\leq n\leq n_{4}}\Big\}.
\end{align*}

Now we are ready for our main concept in this section. Set
\begin{equation}
\begin{split}
\lrr:=&\underset{i\in\{1,2,3\}} {\bigcup }\underbrace{\Big(\underset{\mu\neq\nu\in\Omega} {\bigcup }\big( R^{\mu,\nu}_{i,2}\add R^{\nu,\mu}_{i,2}\big)\cup
\underset{\mu\neq\nu\in\Omega} {\bigcup }\big( R^{\mu,\mu,\nu}_{i,3}\add R^{\mu,\nu,\mu}_{i,3}\add R^{\nu,\mu,\mu}_{i,3}\big)\cup
\underset{\mnodis} {\bigcup } R^{\mu,\nu,\omega}_{i,3}\Big)}_{\text{arity $i$}}\\
&\cup\underbrace{\Big(\underset{\mu\neq\nu\in\Omega} {\bigcup }\big( R^{\mu,\mu,\nu}_{4}\add R^{\mu,\nu,\mu}_{4}\add R^{\nu,\mu,\mu}_{4}\big)\cup
\underset{\mnodis} {\bigcup } R^{\mu,\nu,\omega}_{4}\Big)}_{\text{arity $4$}}.
\end{split}
\mlabel{eq:lin1}
\end{equation}
\begin{defn}
Let $\Omega$ be a nonempty set.
Let $\mscr{P}$ be a unary binary \qc ns operad given in Eq.~\meqref{eq:ubqcop}.
We call the ns operad
$$\lin{\mscr{P}}_\Omega:=\mscr{T}\Big(\bigoplus\limits_{\omega\in\Omega} E_\omega\Big)\Big/\Big\langle \bigcup\limits_{\omega\in\Omega}R_\omega\cup \lrr\Big\rangle $$
the \textbf{linearly compatible operad} of $\spp$ with parameter $\Omega$.
\mlabel{defn:comp1}
\end{defn}
Encoding linearly compatible algebras, we give
\begin{theorem}
Let $\Omega$ be a nonempty set. Let $\spp=\spp(E,R)$ be a \ubqco. A vector space $V$ is
a $\lin{\spp}_\Omega$-algebra if and only if it is an $\Omega$-linearly compatible $\spp$-algebra.
\mlabel{thm:comp}
\end{theorem}

\begin{proof}
Let $\mscr{P}$ be a unary binary \qc ns operad given in Eq.~\meqref{eq:ubqcop}. Let $V$ be a vector space.
On the one hand, by Definition~\mref{defn:comp1}, $V$ is a $\lin{\spp}_\Omega$-algebra if and only if elements of $V$ satisfy the relations in $\bigcup\limits_{\omega\in\Omega}R_\omega\cup \lrr$.

On the other hand, if $V$ is an $\Omega$-linearly compatible $\spp$-algebra.
Then by Proposition~\mref{defn:lcalg2}, $V$ with operations
$$\sum_{\omega\in\Omega}c_\omega P_{\omega,1},\ldots,\sum_{\omega\in\Omega}c_\omega P_{\omega,t}, \sum_{\omega\in\Omega}c_\omega 1_\omega,\ldots, \sum_{\omega\in\Omega}c_\omega s_\omega$$
is a $\spp$-algebra for any choice of $\bfc:=(c_\omega)_{\omega\in \Omega}$ with $c_\omega\in \bfk$,
that is, elements of $V$ satisfy the relations in
$$R_\bfc = \rc{_{1,2}} \sqcup \rc{_{1,3}} \sqcup \rc{_{2,2}} \sqcup \rc{_{2,3}}\sqcup \rc{_{3,2}} \sqcup\rc{_{3,3}} \sqcup\rc{_4} .$$
Explicitly,
\begin{eqnarray*}
\rc{_{1,2}} &=&\biggr\{r_{1,2}^n\Big(\sum_{\omega\in\Omega}c_\omega P_{\omega,k}, \sum_{\omega\in\Omega}c_\omega P_{\omega,\ell}\Big)\,\Big|\,{1\leq n\leq n_{1,2}}\biggr\},\\
\rc{_{1,3}} &=&\biggr\{r_{1,3}^n\Big(\sum_{\omega\in\Omega}c_\omega P_{\omega,k}, \sum_{\omega\in\Omega}c_\omega P_{\omega,\ell},\sum_{\omega\in\Omega}c_\omega P_{\omega,m}\Big)\,\Big|\,{1\leq n\leq n_{1,3}}\biggr\},\\
\rc{_{2,2}} &=&\biggr\{r_{2,2}^n\Big(\sum_{\omega\in\Omega}c_\omega P_{\omega,k},\sum_{\omega\in\Omega}c_\omega i_\omega\Big)\,\Big|\,{1\leq n\leq n_{2,2}}\biggr\},\\
\rc{_{2,3}} &=&\biggr\{r_{2,3}^n\Big(\sum_{\omega\in\Omega}c_\omega P_{\omega,k},\sum_{\omega\in\Omega}c_\omega P_{\omega,\ell},\sum_{\omega\in\Omega}c_\omega i_\omega\Big)\,\Big|\,{1\leq n\leq n_{2,3}}\biggr\},\\
\rc{_{3,2}} &=&\biggr\{r_{3,2}^n\Big(\sum_{\omega\in\Omega}c_\omega i_\omega,\sum_{\omega\in\Omega}c_\omega j_\omega\Big)\,\Big|\,{1\leq n\leq n_{3,2}}\biggr\},\\
\rc{_{3,3}} &=&\biggr\{r_{3,3}^n\Big(\sum_{\omega\in\Omega}c_\omega P_{\omega,k},\sum_{\omega\in\Omega}c_\omega i_\omega,\sum_{\omega\in\Omega}c_\omega j_\omega\Big)\,\Big|\,{1\leq n\leq n_{3,3}}\biggr\},\\
\rc{_{4}} &=&\biggr\{r_{4}^n\Big(\sum_{\omega\in\Omega}c_\omega i_\omega,\sum_{\omega\in\Omega}c_\omega j_\omega,\sum_{\omega\in\Omega}c_\omega p_\omega\Big)\,\Big|\,{1\leq n\leq n_{4}}\biggr\}.\\
\end{eqnarray*}
Thus we are left to verify
$$\Big\langle R_\bfc\Big\rangle=\Big\langle \bigcup\limits_{\omega\in\Omega}R_\omega\cup \lrr\Big\rangle,$$
which, by comparing the arities and Eq.~(\mref{eq:lin1}), is equivalent to the following equations, for $i={1,2,3}$:
\begin{align}
\bfk\rc{_{i,2}} &=\bfk\Big\{ \bigcup\limits_{\omega\in\Omega}R_{\omega,i,2}\cup \bigcup\limits_{\mu\neq\nu\in \Omega} R_{i,2}^{\mu,\nu}\add  R_{i,2}^{\nu,\mu} \Big\}, \label{eq:lastri2}\\
\bfk  \rc{_{i,3}}&=\bfk\Big\{\bigcup\limits_{\omega\in\Omega}R_{\omega,i,3}\cup \underset{\mu\neq\nu\in\Omega} {\bigcup }\big( R^{\mu,\mu,\nu}_{i,3}\add R^{\mu,\nu,\mu}_{i,3}\add R^{\nu,\mu,\mu}_{i,3}\big)\cup
\underset{\mnodis} {\bigcup } R^{\mu,\nu,\omega}_{i,3}\Big\}, \label{eq:lastri3}\\
\bfk  \rc{_4}&=\bfk\Big\{ \bigcup\limits_{\omega\in\Omega}R_{\omega,4}\cup \underset{\mu\neq\nu\in\Omega} {\bigcup }\big( R^{\mu,\mu,\nu}_{4}\add R^{\mu,\nu,\mu}_{4}\add R^{\nu,\mu,\mu}_{4}\big)\cup
\underset{\mnodis} {\bigcup } R^{\mu,\nu,\omega}_{4}\Big\}. \label{eq:lastr4}
\end{align}
Now expanding the set of quadratic relations $\rc{_1}$ by linearity, we obtain
\begin{eqnarray*}
&&\bfk\rc{_{1,2}}\\
&=&\bfk\biggr\{
r_{1,2}^n\big(\sum_{\omega\in\Omega}c_\omega P_{\omega,k}, \sum_{\omega\in\Omega}c_\omega P_{\omega,\ell}\big)
 \,\Big|\,{1\leq n\leq n_{1,2}}\biggr\}\\
&=&\bfk\biggr\{\sum_{\mu,\nu\in\Omega}c_\mu c_\nu r_{1,2}^n\big( P_{\mu,k}, P_{\nu,\ell}\big)
 \,\Big|\,{1\leq n\leq n_{1,2}}\biggr\}\\
&=&\bfk\biggr\{\sum_{\mu\in\Omega}c_\omega^2 r_{1,2}^n\big( P_{\omega,k}, P_{\omega,\ell}\big)
 +\sum_{\mu\neq\nu\in\Omega}c_\mu c_\nu r_{1,2}^n\big( P_{\mu,k}, P_{\nu,\ell}\big)
 \,\Big|\,{1\leq n\leq n_{1,2}}\biggr\}\\
&=&\bfk\biggr\{\sum_{\omega\in\Omega}c_\omega^2 r_{1,2}^n\big( P_{\omega,k}, P_{\omega,\ell}\big)
 +\frac{1}{2}\sum_{\mu\neq\nu\in\Omega}c_\mu c_\nu \Big(r_{1,2}^n\big( P_{\mu,k}, P_{\nu,\ell}\big)+r_{1,2}^n\big( P_{\nu,k}, P_{\mu,\ell}\big)\Big)
 \,\Big|\,{1\leq n\leq n_{1,2}}\biggr\}\\
&=&\bfk \biggr\{r_{1,2}^n\big( P_{\omega,k}, P_{\omega,\ell}\big), r_{1,2}^n\big( P_{\mu,k}, P_{\nu,\ell}\big)+ r_{1,2}^n\big( P_{\nu,k}, P_{\mu,\ell}\big)\,\Big|\,\omega \in \Omega, \mu\neq\nu\in\Omega\biggr\}\\
&& \hspace{5cm} \quad(\text{by the arbitrariness of } c_\omega, c_\mu, c_\nu\in\bfk)\\
&=&\bfk\biggr(\bigcup\limits_{\omega\in\Omega}R_{\omega,1,2}\cup\bigcup\limits_{\mu\neq\nu\in \Omega} R_{1,2}^{\mu,\nu}\add R_{1,2}^{\nu,\mu}\biggr)
\end{eqnarray*}
Thus Eq.~(\mref{eq:lastri2}) holds for $i=1$.
Similarly, expanding $\rc{_{2,2}}$ and $\rc{_{3,2}}$ by linearity, we obtain Eq.~(\mref{eq:lastri2}) for $i=2,3$.

For the cubic relations $\rc{_{2,3}}$, expanding $\rc{_{2,3}}$ by linearity, we have
\begin{eqnarray*}
\bfk\rc{_{2,3}}
&=&\bfk\biggr\{r_{2,3}^n\big(\sum_{\omega\in\Omega}c_\omega P_{\omega,k},\sum_{\omega\in\Omega}c_\omega P_{\omega,\ell},\sum_{\omega\in\Omega}c_\omega i_\omega\big)
\,\Big|\,{1\leq n\leq n_{2,3}}\biggr\}\\
&=&\bfk\biggr\{\sum_{\mu,\nu,\omega\in\Omega}c_\mu c_\nu c_\omega r_{2,3}^n\big(P_{\mu,k},P_{\nu,\ell},i_\omega\big)
\,\Big|\,{1\leq n\leq n_{2,3}}\biggr\}\\
&=&\bfk\biggr\{\sum_{\omega\in\Omega}c_\omega ^3 r_{2,3}^n\big(P_{\omega,k},P_{\omega,\ell},i_\omega\big)+\sum_{\mnodis}c_\mu c_\nu c_\omega r_{2,3}^n\big(P_{\mu,k},P_{\nu,\ell},i_\omega\big)\\
&&+\sum_{\mu\neq\nu\in\Omega}c_\mu ^2 c_\nu \Big( r_{2,3}^n\big(P_{\mu,k},P_{\mu,\ell},i_\nu\big)+ r_{2,3}^n\big(P_{\mu,k},P_{\nu,\ell},i_\mu\big)
+ r_{2,3}^n\big(P_{\nu,k},P_{\mu,\ell},i_\mu\big)\Big)
\,\Big|\,{1\leq n\leq n_{2,3}}\biggr\}\\
&=&\bfk \biggr\{r_{2,3}^n\big(P_{\omega,k},P_{\omega,\ell},i_\omega\big),\,
r_{2,3}^n\big(P_{\mu,k},P_{\mu,\ell},i_\nu\big)+ r_{2,3}^n\big(P_{\mu,k},P_{\nu,\ell},i_\mu\big)
+ r_{2,3}^n\big(P_{\nu,k},P_{\mu,\ell},i_\mu\big),\,\\
&&r_{2,3}^n\big(P_{\mu,k},P_{\nu,\ell},i_\omega\big)\,\Big|\,\mu,\nu,\omega\in \Omega \text{ distinct} \biggr\}\quad (\text{by the arbitrariness of } c_\mu, c_\nu,c_\omega\in\bfk)\\
&=&\bfk \biggr(\bigcup\limits_{\omega\in\Omega}R_{\omega,2,3}\cup \underset{\mu\neq\nu\in\Omega} {\bigcup }\big( R^{\mu,\mu,\nu}_{2,3}\add R^{\mu,\nu,\mu}_{2,3}\add R^{\nu,\mu,\mu}_{2,3}\big)\cup \underset{\mnodis} {\bigcup }  R^{\mu,\nu,\omega}_{2,3}\biggr).
\end{eqnarray*}
Therefore
Eq.~(\mref{eq:lastri3}) also holds, for $i=2$. Similarly, expanding the cubic relations $\rc{_{2,3}}$, $\rc{_{3,3}}$ and $\rc{_{4}}$ by linearity, we obtain Eq.~(\mref{eq:lastri2}) for $i=2,3$, and Eq.~\meqref{eq:lastr4}. This completes the proof.
\end{proof}

\subsection{Linear compatibility by the Manin black square product} 
We now establish the relationship between the linear compatibility and the Manin black square products of ns operads.
\begin{defn}	\mlabel{defn:manin}
	 (\mcite{EFG,Val})
	Let $\spp=\stt(E)/\langle R\rangle$ and $\mscr{Q}=\stt(F)/\langle S\rangle$ be two finite generated binary quadratic ns opeards with relations
	\begin{equation*}
		R= ~\biggr\{\sum_{i,j\in E}{\kappa}_{i,j}^{m,1}
		\treey{\cdlr{ol}
			\foreach \i/\j in {ol/$i$,o/$j$} {\node[above] at (\i) {\j};}}
		+{\kappa}_{i,j}^{m,2}\treey{\cdlr{or}
			\foreach \i/\j in {or/$j$,o/$i$} {\node[above] at (\i) {\j};}}
		\quad \Big| \,{1\leq m\leq r}\biggr\}
	\end{equation*}
	and
	\begin{equation*}
		S= ~\biggr\{\sum_{k,\ell\in F}{\kappa'}_{k,\ell}^{n,1}
		\treey{\cdlr{ol}
			\foreach \i/\j in {ol/$k$,o/$\ell$} {\node[above] at (\i) {\j};}}
		+{\kappa'}_{k,\ell}^{n,2}\treey{\cdlr{or}
			\foreach \i/\j in {or/$\ell$,o/$k$} {\node[above] at (\i) {\j};}}
		\quad \Big| \,{1\leq n\leq s}\biggr\}.
	\end{equation*}
	\begin{enumerate}
		\item 		\mlabel{item:maninb}
		The {\bf Manin black product} of $\spp$ and $\mscr{Q}$ is the operad
		$
			\spp \blacksquare \mscr{Q}:=\stt(E\otimes F)/\langle R\blacksquare S\rangle,
		$
		where
		\begin{equation}
			R\blacksquare S:= ~\biggr\{\sum_{i,j\in E}\sum_{k,\ell\in F}{\kappa}_{i,j}^{m,1}{\kappa'}_{k,\ell}^{n,1}
			\treey{\cdlr{ol}\node  at (-0.6,0.1) {$i\otimes k$};\node  at (0.4,0) {$j\otimes \ell$};}
			-{\kappa}_{i,j}^{m,2}{\kappa'}_{k,\ell}^{n,2}\treey{\cdlr{or}
				\node  at (0.6,0.1) {$j\otimes \ell$};\node  at (-0.4,-0.1) {$i\otimes k$};}
			\quad \Big| \,{1\leq m\leq r},\,{1\leq n\leq s}\biggr\}.
\mlabel{eq:maninb}
		\end{equation}
		\item		\mlabel{item:maninw}
		The {\bf Manin white product} of $\spp$ and $\mscr{Q}$ is the operad
		$
			\spp \square \mscr{Q}:=\stt(E\otimes F)/\langle R\square S\rangle,
			\mlabel{eq:maninw}
		$
		where
		\begin{align*}
			R\square S:=&~\biggr\{\sum_{i,j\in E}\sum_{k,\ell\in F}{\kappa}_{i,j}^{m,1}
			\treey{\cdlr{ol}\node  at (-0.6,0.1) {$i\otimes k$};\node  at (0.4,0) {$j\otimes \ell$};}
			-{\kappa}_{i,j}^{m,2}\treey{\cdlr{or}
				\node  at (0.6,0.1) {$j\otimes \ell$};\node  at (-0.4,-0.1) {$i\otimes k$};}
			\quad \Big| \,{1\leq m\leq r}\biggr\}\\
			&\cup\biggr\{\sum_{i,j\in E}\sum_{k,\ell\in F}{\kappa'}_{k,\ell}^{n,1}
			\treey{\cdlr{ol}\node  at (-0.6,0.1) {$i\otimes k$};\node  at (0.4,0) {$j\otimes \ell$};}
			-{\kappa'}_{k,\ell}^{n,2}\treey{\cdlr{or}
				\node  at (0.6,0.1) {$j\otimes \ell$};\node  at (-0.4,-0.1) {$i\otimes k$};}
			\quad \Big| \,{1\leq n\leq s}\biggr\}.
		\end{align*}
	\end{enumerate}
\end{defn}

The following result shows that linear compatibility can be achieved by taking a black square product.
\begin{prop}
\mlabel{prop:maninbl}
Let $\Omega$ be a nonempty finite set. Let $\spp$ be a finitely generated binary quadratic ns operad.
Then
$$\lin{\spp}_\Omega\cong \lin{\as}_\Omega\blacksquare\spp. $$
\end{prop}
\begin{proof}
By Definition~\mref{defn:comp1}, for the associative operad $ \as =\stt(E)/\langle R\ \rangle$ in Example~\mref{ex:nsop}~\mref{exam:0}, we have
\begin{align*}
\lin{\as}_\Omega
=\mscr{T}\Big(\bigoplus\limits_{\omega\in\Omega} E_\omega\Big)\Big/\Big\langle \bigcup\limits_{\omega\in\Omega}R_\omega\cup \lrr\Big\rangle
=\mscr{T}\Big(\bigoplus\limits_{\omega\in\Omega} E_\omega\Big)\Big/\Big\langle \bigcup\limits_{\omega\in\Omega}R_\omega\cup \Big(\underset{\mu\neq\nu\in\Omega} {\bigcup } R^{\mu,\nu}\add R^{\nu,\mu}\Big)\Big\rangle,
\end{align*}
where
$$\bigoplus\limits_{\omega\in\Omega} E_\omega=
\bigoplus\limits_{\omega\in\Omega}\bfk\treey{\cdlr{o}\node  at (0,0.35) {$\omega$}; }$$
and
\begin{equation}
 \bigcup\limits_{\omega\in\Omega}R_\omega\cup \Big(\underset{\mu\neq\nu\in\Omega} {\bigcup } R^{\mu,\nu}\add R^{\nu,\mu}\Big)
 =\bigcup\limits_{\omega\in\Omega}\biggr\{\treey{\cdlr{ol}\foreach \i/\j in {ol/$\omega$,o/$\omega$} {\node [above] at (\i) {\j};}}
 -\treey{\cdlr{or}\foreach \i/\j in {or/$\omega$,o/$\omega$} {\node[above] at (\i) {\j};}}\biggr\}
  \cup \underset{\mu\neq\nu\in\Omega} {\bigcup }\biggr\{
\treey{\cdlr{ol}\foreach \i/\j in {ol/$\mu$,o/$\nu$} {\node [above] at (\i) {\j};}}
+ \treey{\cdlr{ol}\foreach \i/\j in {ol/$\nu$,o/$\mu$} {\node [above] at (\i) {\j};}}
 -\treey{\cdlr{or}\foreach \i/\j in {or/$\mu$,o/$\nu$} {\node[above] at (\i) {\j};}}
 -\treey{\cdlr{or}\foreach \i/\j in {or/$\nu$,o/$\mu$} {\node[above] at (\i) {\j};}}\biggr\}.
 \mlabel{eq:linassr}
\end{equation}
Let $\spp=\stt(F)/\langle S \rangle$  with
\begin{equation*}
	F :=F_2:=\bfk\left\{
\treey{\cdlr{o}\node  at (0,0.25) {$1$};},
\treey{\cdlr{o}\node  at (0,0.25) {$2$};},~... ~,
\treey{\cdlr{o}\node  at (0,0.25) {$s$};}
\right\}, \quad
S:= S_{3,2}:=~\biggr\{\sum_{1\leq i,j\leq s}\ka{i}{j}{1}
\treey{\cdlr{ol}
\foreach \i/\j in {ol/$i$,o/$j$} {\node[above] at (\i) {\j};}}
+\ka{i}{j}{2}\treey{\cdlr{or}
\foreach \i/\j in {or/$j$,o/$i$} {\node[above] at (\i) {\j};}}
\quad \Big| \,{1\leq d\leq s}\biggr\}.
\end{equation*}
By Eq.~\meqref{eq:opcirc}, we obtain
\begin{equation}
\Big(\bigoplus\limits_{\omega\in\Omega} E_\omega\Big) \otimes F
=\Big(\bigoplus\limits_{\omega\in\Omega}\bfk\treey{\cdlr{o}\node  at (0,0.35) {$\omega$}; }\Big)
\otimes\Big(\bigoplus\limits_{1\leq  i \leq s}\bfk\treey{\cdlr{o}\node  at (0,0.35) {$i$}; }\Big)\\
=\bigoplus\limits_{\omega\in\Omega,1\leq  i \leq s}\bfk\treey{\cdlr{o}\node  at (0.55,0) {$\omega\otimes i$}; }\\
\cong \bigoplus\limits_{\omega\in\Omega,1\leq  i \leq s}\bfk\treey{\cdlr{o}\node  at (0.4,0) {$i_\omega$}; }\\
=\bigoplus\limits_{\omega\in\Omega} F_\omega
\mlabel{eq:wqpiso}
\end{equation}
and by Definition~\mref{defn:manin}~\mref{item:maninb},
\begin{eqnarray*}
&&\Big(\bigcup\limits_{\omega\in\Omega}R_\omega\cup\underset{\mu\neq\nu\in\Omega} {\bigcup } R^{\mu,\nu}\add R^{\nu,\mu}\Big)\blacksquare S\\
&=&\bigcup\limits_{\omega\in\omega}\big(R_\omega\blacksquare S\big)\cup\underset{\mu\neq\nu\in\omega} {\bigcup }\Big( \big (R^{\mu,\nu}\add R^{\nu,\mu}\big)\blacksquare S\Big)\\
&=&\bigcup\limits_{\omega\in\Omega}\biggr\{\treey{\cdlr{ol}\foreach \i/\j in {ol/$\omega$,o/$\omega$} {\node [above] at (\i) {\j};}}
 -\treey{\cdlr{or}\foreach \i/\j in {or/$\omega$,o/$\omega$} {\node[above] at (\i) {\j};}}\biggr\}
 \blacksquare \biggr\{\sum_{1\leq i,j\leq s}\ka{i}{j}{1}
\treey{\cdlr{ol}
\foreach \i/\j in {ol/$i$,o/$j$} {\node[above] at (\i) {\j};}}
+\ka{i}{j}{2}\treey{\cdlr{or}
\foreach \i/\j in {or/$j$,o/$i$} {\node[above] at (\i) {\j};}}
\quad \Big| \,{1\leq d\leq s}\biggr\}\\
&& \cup \underset{\mu\neq\nu\in\Omega} {\bigcup }\biggr\{
\treey{\cdlr{ol}\foreach \i/\j in {ol/$\mu$,o/$\nu$} {\node [above] at (\i) {\j};}}
+ \treey{\cdlr{ol}\foreach \i/\j in {ol/$\nu$,o/$\mu$} {\node [above] at (\i) {\j};}}
 -\treey{\cdlr{or}\foreach \i/\j in {or/$\mu$,o/$\nu$} {\node[above] at (\i) {\j};}}
 -\treey{\cdlr{or}\foreach \i/\j in {or/$\nu$,o/$\mu$} {\node[above] at (\i) {\j};}}\biggr\}
 \blacksquare
 \biggr\{\sum_{1\leq i,j\leq s}\ka{i}{j}{1}
\treey{\cdlr{ol}
\foreach \i/\j in {ol/$i$,o/$j$} {\node[above] at (\i) {\j};}}
+\ka{i}{j}{2}\treey{\cdlr{or}
\foreach \i/\j in {or/$j$,o/$i$} {\node[above] at (\i) {\j};}}
\quad \Big| \,{1\leq d\leq s}\biggr\}\ (\text{by Eq.~}\meqref{eq:linassr})\\
&\cong&\bigcup\limits_{\omega\in\Omega}\biggr\{\sum_{1\leq i,j\leq s}\ka{i}{j}{1}
\treey{\cdlr{ol}
\foreach \i/\j in {ol/$i_\omega$,o/$j_\omega$} {\node[above] at (\i) {\j};}}
+\ka{i}{j}{2}\treey{\cdlr{or}
\foreach \i/\j in {or/$j_\omega$,o/$i_\omega$} {\node[above] at (\i) {\j};}}
\quad \Big| \,{1\leq d\leq s}\biggr\}\\
&& \cup \underset{\mu\neq\nu\in\Omega} {\bigcup }\biggr\{\sum_{1\leq i,j\leq s}\Big(\big(\ka{i}{j}{1}
\treey{\cdlr{ol}\foreach \i/\j in {ol/$i_\mu$,o/$j_\nu$} {\node [above] at (\i) {\j};}}+\ka{i}{j}{2}\treey{\cdlr{or}\foreach \i/\j in {or/$j_\nu$,o/$i_\mu$} {\node[above] at (\i) {\j};}}\big)+\big(\ka{i}{j}{1}
\treey{\cdlr{ol}\foreach \i/\j in {ol/$i_\nu$,o/$j_\mu$} {\node [above] at (\i) {\j};}}+\ka{i}{j}{2}\treey{\cdlr{or}\foreach \i/\j in {or/$j_\mu$,o/$i_\nu$} {\node[above] at (\i) {\j};}}\big)\Big)
\quad \Big| {1\leq d\leq s}\biggr\}\\
&& \hspace{3.5cm} \qquad \quad(\text{by Eqs.}~\meqref{eq:maninb} \text{ and }~\meqref{eq:wqpiso})\\
&=& \bigcup\limits_{\omega\in\Omega}S_{\omega,3,2}\cup\underset{\mu\neq\nu\in\Omega} {\bigcup } S_{3,2}^{\mu,\nu}\add S_{3,2}^{\nu,\mu}\quad\quad(\text{by Eqs.}~\meqref{eq:pr3b}\text{ and }~\meqref{eq:spr123}).
\end{eqnarray*}
Hence
\begin{eqnarray*}
 \lin{\as}_\Omega\blacksquare\spp
&=&\stt\Big(\big(\bigoplus\limits_{\omega\in\Omega} E_\omega \big)\otimes F\Big)\Big/\Big\langle \Big(\bigcup\limits_{\omega\in\Omega}R_\omega\cup\underset{\mu\neq\nu\in\Omega} {\bigcup } R^{\mu,\nu}\add R^{\nu,\mu}\Big)\blacksquare S\Big\rangle\\
&\cong&\stt\Big(\bigoplus\limits_{\omega\in\Omega} F_\omega \Big)\Big/\Big\langle \bigcup\limits_{\omega\in\Omega}S_{\omega,3,2}\cup\underset{\mu\neq\nu\in\Omega} {\bigcup } S_{3,2}^{\mu,\nu}\add S_{3,2}^{\nu,\mu}\Big\rangle\\
&=&\lin{\spp}_\Omega,
\end{eqnarray*}
giving the desired equality.
\end{proof}

\section{The matching compatibility and Koszul self duality}\mlabel{sec:matcom}
This section is devoted to the operadic study of another compatibility condition of algebraic structures carrying multiple copies of the same operations. It will be called the matching compatibility, which has stronger conditions than the linear compatibility and has a self dual property.

\subsection{Matching compatibility}\mlabel{sec:mat1}

By splitting Eq.~(\mref{eq:lin1}) according to arities and degrees, we define
\begin{equation}
\mrr:=\underset{i\in\{1,2,3\}} {\bigcup }\Big(\underbrace{\underset{\mu\neq\nu\in\Omega} {\bigcup } R^{\mu,\nu}_{i,2}
\cup\underset{\mu,\nu,\omega\in\Omega \atop \text{  not all equal}} {\bigcup } R^{\mu,\nu,\omega}_{i,3}}_{\text{arity $i$}}\Big)
\cup \underbrace{\underset{\mu,\nu,\omega\in\Omega \atop \text{  not all equal}} {\bigcup } R^{\mu,\nu,\omega}_{4}}_{\text{arity 4}}
\mlabel{eq:mat}
\end{equation}
\begin{defn}
Let $\Omega$ be a nonempty set.
Let $\spp=\mscr{T}(E)/\langle R\rangle $ be a unary binary \qc ns operad given in Eq.~\meqref{eq:ubqcop}.
We call the ns operad
\begin{equation*}
\mat{\spp}_\Omega:=\mscr{T}\Big(\bigoplus\limits_{\omega\in\Omega} E_\omega\Big)\Big/\Big\langle \bigcup\limits_{\omega\in\Omega}R_\omega\cup \mrr\Big\rangle
\end{equation*}
the \textbf{matching operad} of $\mscr{P}$ with parameter set $\Omega$.
\mlabel{defn:mat1}
\end{defn}
By Eqs.~\meqref{eq:recb} and~\meqref{eq:spr123}, we have
$$ \bigcup\limits_{\omega\in\Omega}R_\omega\cup \mrr=\underset{\mu,\nu,\omega\in\Omega\atop i\in\{1,2,3\}} {\bigcup }\Big( R^{\mu,\nu}_{i,2}\cup R^{\mu,\nu,\omega}_{i,3}\cup R^{\mu,\nu,\omega}_4\Big).$$
Thus
$$\mat{\spp}_\Omega=\mscr{T}\Big(\bigoplus\limits_{\omega\in\Omega} E_\omega\Big)\Big/\Big\langle\underset{\mu,\nu,\omega\in\Omega\atop i\in\{1,2,3\}} {\bigcup }\Big( R^{\mu,\nu}_{i,2}\cup R^{\mu,\nu,\omega}_{i,3}\cup R^{\mu,\nu,\omega}_4\Big)\Big\rangle.$$

Since $\langle\lrr\rangle \subseteq \langle \mrr\rangle$, we have

\begin{prop}
Let $\Omega$ be a nonempty set.
Let $\spp$ be a unary binary \qc ns operad given in Eq.~\meqref{eq:ubqcop}.
Then there is an epimorphism of ns operads
$$\lin{\spp}_\Omega \longrightarrow \mat{\spp}_\Omega.$$
In other words, any $\mat{\spp}_\Omega$-algebra is a $\lin{\spp}_\Omega$-algebra.
\mlabel{prop:matlin0}
\end{prop}

We give some examples on the level of algebras.
\begin{exam} 	\mlabel{ex:rbmat}
Let $\Omega$ be a finite set.
\begin{enumerate}
\item
\mlabel{ex:rbmatb}
When $\spp$ is the operad of associative algebras, $\mat{\spp}_\Omega$ gives the matching associative algebra~\mcite{ZBG,Zi}, defined to be a vector space $A$ equipped with a family of
	binary operations $\cdot_\omega: A\otimes A\to A,$ satisfying the matching associativity
\begin{equation}
\mlabel{eq:mass}
	(x\cdot_\alpha y)\cdot_\beta z=x\cdot_\alpha( y\cdot_\beta z),\,\text{ for all }\, x,y,z\in A, \alpha,\beta\in\Omega.
\end{equation}
\item  When $\spp$ is the operad of Rota-Baxter algebras of weight zero, $\mat{\spp}_\Omega$ gives the following algebraic structure with multiple copies $\cdot_\alpha$ of the multiplication and multiple copies $P_\alpha$ of the Rota-Baxter operator, satisfying Eq.~\meqref{eq:mass} and
$$P_\alpha(x)\cdot_\gamma P_\beta(y)=P_\alpha(x\cdot_\gamma P_\beta(y))+P_\beta(P_\alpha(x)\cdot_\gamma y)\,\text{ for all }\, x,y\in A, $$
where $\alpha,\beta,\gamma\in\Omega$ with exactly two of which to be the same.
\mlabel{ex:rbmatc}
\item
In the above example, if the multiplications $\cdot_\alpha$ are taken to be the same, we recover the notion of matching Rota-Baxter algebras introduced in~\mcite{ZGG20}.
		\mlabel{ex:rbmata}
\item
\mlabel{it:dend}
Taking $\spp$ to be the operad $\dend$ of dendriform algebras. A $\mat{\dend}_\Omega$-algebra is a matching dendriform algebra $(D, \{\prec_\alpha, \succ_\beta\,|\,\alpha,\beta\in\Omega\})$, characterized by the relations
\begin{align*}
	(x\prec_\alpha y)\prec_\beta z&=x\prec_\alpha(y\prec_\beta z)+x\prec_\alpha(y\succ_\beta z),\\
	(x\succ_\alpha y)\prec_\beta z&=x\succ_\alpha(y\prec_\beta z),\\
	(x\prec_\alpha y)\succ_\beta z+(x\succ_\alpha y)\succ_\beta z&=x\succ_\alpha(y\succ_\beta z), \ \text{ for all } x, y, z\in D.
\end{align*}
Then it is immediately checked that $(D, \{\ast_\alpha:=\prec_\alpha+\succ_\alpha\,|\,\alpha\in\Omega\})$ is a matching associative algebra in Item~\mref{ex:rbmatb}.
This relationship of matching compatibility with the splitting of operads~\mcite{BBGN,GK,PBGN} holds for any binary quadratic ns operad.
We note that this notion of matching dendriform algebra is different from the one defined in~\mcite{ZGG20}. For example, instead of the first equation, the notion in~\mcite{ZGG20} has the equation
$$(x\prec_\alpha y)\prec_\beta z=x\prec_\alpha(y\prec_\beta z)+x\prec_\beta(y\succ_\alpha z).$$
	\end{enumerate}
\end{exam}

\subsection{The Koszul dual of unary binary quadratic operads}\mlabel{sec:kosdual}
The Koszul duality for binary operads was given by Ginzburg and Kapranov~\mcite{GK94} for binary quadratic operads and by Getzler~\mcite{Get95} for binary operads. See also~\mcite{Fre, LV}.
\begin{defn}
	\begin{enumerate}
	\mlabel{defn:kos}
		\item  \mlabel{item:quaop}
	An {\bf quadratic cooperad}  $\mathscr{C}(E,R)$ associated to the cogenerators $E$ and the corelations $R$ is a sub-cooperad of the cofree cooperad
		$\mathscr{T}^n(E)$ such that the composite
		$$\mathscr{C}(E,R)\rightarrowtail \mathscr{T}^n(E) \twoheadrightarrow \mathscr{T}^n(E)^{(2)}/R$$
		is 0.
		\item  	
		Let $\mathscr{P}=\mathscr{T}(E)/\langle R\rangle$ be a quadratic operad. 	Let $s$ be the suspension.
		Define the {\bf Koszul dual cooperad } of $\spp$ to be the quadratic cooperad
		$$\mathscr{P}^\text{!`}:= \mathscr{C}(sE,s^2R).$$
		\item
		Let $\mathscr{P}=\mathscr{T}(E)/\langle R\rangle$ be a quadratic operad.  Let $\mathscr{S}^n$ be the cooperad defined by  $\mathscr{S}^c_n:=\text{Hom}\big((s\bfk)^n ,s\bfk\big)$. 
		 Define
		the {\bf Koszul dual operad} of $\spp$ to be the linear dual
		$$\mathscr{P}^!:= \Big(\mathscr{S}^c\underset{ {\rm H} }{\otimes}\mathscr{P}^\text{!`} \Big)^*.$$
	\end{enumerate}
\end{defn}

For the operad of differential algebras, the Koszul dual was explicitly computed by Loday in~\mcite{Lod10}. Generalizing the approach, we obtain

\begin{theorem}
	For any finitely generated unary binary quadratic ns operad $\spp = \mathscr{T}(E)/\langle R\rangle $ with
$$R=R_1\sqcup R_{2}\sqcup R_3:=R_{1,2}\sqcup R_{2,2}\sqcup R_{3,2},$$
  its Koszul dual operad is given
	by
	\begin{equation*}
		\spp^!= \mathscr{T}(E^\ast)/\langle R^\perp\rangle,
	\end{equation*}
	where $R^\perp:=\big\{R_1^\perp, R_{2}^\perp, R_3^\perp\big\}$, while
	$R_1^\perp$, $R_{2}^\perp$ and $R_3^\perp$ are respectively the orthogonal subspaces of $R_1, R_{2}$ and $R_3$ in $E_1\otimes E_1$, $E_1\otimes E_2\oplus E_2\otimes E_1\oplus E_2\otimes E_1$ and $E_2\otimes E_2\oplus E_2\otimes E_2$.
	\mlabel{lem:koszul}
\end{theorem}
We give some notations before the proof. Since $\spp$ is a unary binary quadratic ns operad, the generator is $E=\{0,E_1, E_2,\ldots\}$ and
\begin{eqnarray*}
	&&\stt(E)^{(2)}\\
	&=&\Big\{\stt(E)^{(2)}_1, \stt(E)^{(2)}_2, \stt(E)^{(2)}_3\Big\}\\
	&=&\Big\{\bfk\big\{\alpha\circ_1\beta\,|\,\alpha,\beta\in E_1\} ,
	\bfk\big\{\alpha\circ_1\mu,\mu\circ_1\alpha,\mu\circ_2\alpha\,|\,\alpha\in E_1, \mu\in E_2\},
	\bfk\big\{\mu\circ_1\nu,\mu\circ_2\nu\,|\, \mu,\nu\in E_2\}\Big\}\\
	&=&\Big\{E_1\otimes E_1, E_1\otimes E_2\oplus E_2\otimes E_1\oplus E_2\otimes E_1, E_2\otimes E_2\oplus E_2\otimes E_2\Big\}.
\end{eqnarray*}
Thus
$$\Big(\stt(E)^{(2)}\Big)^\ast=\Big\{\Big(\stt(E)^{(2)}_1\Big)^\ast, \Big(\stt(E)^{(2)}_2\Big)^\ast, \Big(\stt(E)^{(2)}_3\Big)^\ast\Big\}.$$
We identify $\stt(E^\ast)^{(2)}$ with the dual of $\stt(E)^{(2)}$ by means of the
nondegenerate bilinear form
\begin{equation}
\begin{aligned}
\pair{-,-}_i: \stt(E^\ast)^{(2)}_i&\otimes \stt(E)^{(2)}_i\longrightarrow \bfk,\,\text{ for } i=1,2,3,\\
\pair{\alpha^\ast\circ_1\beta^\ast, \alpha\circ_1\beta}_1&:=\alpha^\ast(\alpha)\beta^\ast(\beta)\in\bfk,\quad
\pair{\alpha^\ast\circ_1\mu^\ast, \alpha\circ_1\mu}_2:=-\alpha^\ast(\alpha)\mu^\ast(\mu)\in\bfk, \\
\pair{\mu^\ast\circ_1\alpha^\ast, \mu\circ_1\alpha}_2&:=\mu^\ast(\mu)\alpha^\ast(\alpha)\in\bfk, \quad
\pair{\mu^\ast\circ_2\alpha^\ast, \mu\circ_2\alpha}_2:=\mu^\ast(\mu)\alpha^\ast(\alpha)\in\bfk,\\
\pair{\mu^\ast\circ_1\nu^\ast, \mu\circ_1\nu}_3&:=\mu^\ast(\mu)\nu^\ast(\nu)\in\bfk,\quad
\pair{\mu^\ast\circ_2\nu^\ast, \mu\circ_2\nu}_3:=-\mu^\ast(\mu)\nu^\ast(\nu)\in\bfk,\\
\pair{-,-}_i&:=0,\,\text{ otherwise },\, \text{ where }\, \alpha,\beta\in E_1, \mu,\nu\in E_2.
\end{aligned}
\mlabel{eq:midiso}
\end{equation}
Then, for $i=1,2,3$, the orthogonal subspace $R_i^\perp\subseteq\stt(E^\ast)^{(2)}_i$ is
\begin{equation}
	R_i^\perp=\Big\{x\in\stt(E^\ast)^{(2)}_i \,\Big|\, \pair{x,R_i}_i=0\Big\}.
	\mlabel{eq:rperp}
\end{equation}
\begin{proof} {\em (of Theorem~\mref{lem:koszul})}
	Since $\mathscr{S}^c_n=\text{Hom}\big((s\bfk)^n ,s\bfk\big)$, the cogenerators of $\mathscr{S}^c$ in arity 1 and 2 are $\id$ and $s^{-1}$, respectively.
	Thus the cogenerators of the cooperad $\mathscr{S}^c\underset{ {\rm H} }{\otimes}\mathscr{P}^\text{!`}$
	is $E$ with degree 0. Notice that $R$ is an arity graded subspace of $\stt(E)^{(2)}$ with
	\begin{align*}
		\stt(E)^{(2)}_1 &= E_1\otimes E_1,
\\
		\stt(E)^{(2)}_2 &= E_1\otimes E_2\oplus E_2\otimes E_1\oplus E_2\otimes E_1,
\\
		\stt(E)^{(2)}_3 &= E_2\otimes E_2\oplus E_2\otimes E_2.
	\end{align*}
From Definition~\mref{defn:kos}, we have
	$$
		\mathscr{P}^\text{!`}= \mathscr{C}(sE,s^2R)\,\text{ and }\,
		\mathscr{S}^c\underset{ {\rm H} }{\otimes}\mathscr{P}^\text{!`}= \mathscr{C}\big(\big\{E_1,E_2\big\},R\big).$$
	For the cooperad $\mathscr{C}\big(\big\{E_1,E_2\big\},R\big)$, the associated linear dual operad is the quadratic operad $\spp(E^\ast, R^\perp)$, where $R^\perp$ is obtained as follows:
	$$\xymatrix{
		{\rm Ker}(\pi) \ar[d]_{\cong} \ar[r]^{i} & \big(\stt(E)^{(2)}\big)^\ast \ar[d]_{\phi\,\cong} \ar[r]^{\quad\pi} & (R)^\ast  \\
		R^\perp \ar[r]^{} & \stt(E^\ast)^{(2)}.    }
	$$
	Here the isomorphism $\phi$ is induced by  the scalar product in Eq.~\meqref{eq:midiso}.
\end{proof}
As an example, we compute the Koszul dual of the operad of $\Delta$-differential algebras in Example~\mref{ex:nsop}~\mref{item:ddas}. Taking $\Delta$ to be a singleton gives~\cite[Proposition~7.2]{Lod10}.
\begin{prop}\mlabel{prop:kdualdda}
Let $\Delta$ be a nonempty finite set. Then the Koszul dual operad of the operad of $\Delta$-differential algebra is presented by generators
$$E_1^\ast = ~\biggr\{\bfk\treeyy[scale=0.8]{\cdu{o} \node  at (0,0) [scale=0.6]{$\bullet$};\node at (0.3,0) {$\delta^\ast$};}\,\big|\,
 \delta\in\Delta\biggr\},\quad
	E_2^\ast = \bfk\treey{\cdlr{o}\node at  (0.1,0.25)  {$\mu^\ast$}; },$$
and relations
\begin{align*}
R_1^\perp=\left\{\left .
\treeyy{\cdu{o} \node  at (0,-0.2) [scale=0.6]{$\bullet$};\node at (0.3,-0.2) {$\delta_1^\ast$};
		\node  at (0,0.2) [scale=0.6]{$\bullet$};\node at (0.3,0.2) {$\delta_2^\ast$};}
+\treeyy{\cdu{o} \node  at (0,-0.2) [scale=0.6]{$\bullet$};\node at (0.3,-0.2) {$\delta_2^\ast$};
		\node  at (0,0.2) [scale=0.6]{$\bullet$};\node at (0.3,0.2) {$\delta_1^\ast$};}
\,\right |\,\delta_1,\delta_2\in\Delta\right\},\quad
~R_2^\perp=\left\{\left .
	\treey{\cdlr[0.8]{o} \cdl{ol}\cdr{or}\zhds{o/b}
		\node at (0.3,-0.2) {$\delta^\ast$};\node at  (0.1,0.25)  {$\mu^\ast$};}
	-\treey{\cdlr[0.8]{o} \cdl{ol}\cdr{or}
		\node at (ol) [scale=0.6]{$\bullet$};\node at (-0.5,0.1) {$\delta^\ast$};\node at  (0.1,0.25)  {$\mu^\ast$};},\quad
	\treey{\cdlr[0.8]{o} \cdl{ol}\cdr{or}\zhds{o/b}
		\node at (0.3,-0.2) {$\delta^\ast$};\node at  (0.1,0.25)  {$\mu^\ast$};}
	-\treey{\cdlr[0.8]{o} \cdl{ol}\cdr{or}
		\node at (or) [scale=0.6]{$\bullet$};\node at (0.5,0.1) {$\delta^\ast$};\node at  (0.1,0.25)  {$\mu^\ast$};}\,\right|\,
 \delta\in\Delta\right\},
\end{align*}
$$~R_3^\perp=\left\{\treey{\cdlr{ol}\node at  (0.25,-0.15)  {$\mu^\ast$};\node at  (-0.25,0.5)  {$\mu^\ast$};}-\treey{\cdlr{or}\node at  (0.25,-0.15)  {$\mu^\ast$};\node at  (0.4,0.6)  {$\mu^\ast$};}\right\}.
$$
\end{prop}

\begin{proof}
	By Example~\mref{ex:nsop}~\mref{item:ddas}, Theorem~\mref{lem:koszul} and Eq.~\meqref{eq:rperp},
	we have
	\begin{align*}
\Big\langle\treeyy{\cdu{o} \node  at (0,-0.2) [scale=0.6]{$\bullet$};\node at (0.3,-0.2) {$\delta_1^\ast$};
		\node  at (0,0.2) [scale=0.6]{$\bullet$};\node at (0.3,0.2) {$\delta_2^\ast$};}
+\treeyy{\cdu{o} \node  at (0,-0.2) [scale=0.6]{$\bullet$};\node at (0.3,-0.2) {$\delta_2^\ast$};
		\node  at (0,0.2) [scale=0.6]{$\bullet$};\node at (0.3,0.2) {$\delta_1^\ast$};},
\treeyy{\cdu{o} \node  at (0,-0.2) [scale=0.6]{$\bullet$};\node at (0.3,-0.2) {$\delta_1$};
		\node  at (0,0.2) [scale=0.6]{$\bullet$};\node at (0.3,0.2) {$\delta_2$};}
-\treeyy{\cdu{o} \node  at (0,-0.2) [scale=0.6]{$\bullet$};\node at (0.3,-0.2) {$\delta_2$};
		\node  at (0,0.2) [scale=0.6]{$\bullet$};\node at (0.3,0.2) {$\delta_1$};}
\Big\rangle_1=\pair{\delta_1^\ast\circ_1\delta_2^\ast, \delta_1\circ_1\delta_2}_1-\pair{\delta_2^\ast\circ_1\delta_1^\ast,\delta_2\circ_1\delta_1}_1
=\delta_1^\ast(\delta_1)\delta_2^\ast(\delta_2)-\delta_2^\ast(\delta_2)\delta_1^\ast(\delta_1)=0,
	\end{align*}
and	
\begin{align*}
		\Bigg\langle
		\treey{\cdlr[0.8]{o} \cdl{ol}\cdr{or}\zhds{o/b}
			\node at (0.3,-0.2) {$\delta^\ast$};\node at  (0.1,0.25)  {$\mu^\ast$};}
		-\treey{\cdlr[0.8]{o} \cdl{ol}\cdr{or}
			\node at (ol) [scale=0.6]{$\bullet$};\node at (-0.5,0.1) {$\delta^\ast$};\node at  (0.1,0.25)  {$\mu^\ast$};},
		\treey{\cdlr[0.8]{o} \cdl{ol}\cdr{or}\zhds{o/b}
			\node at (0.3,-0.2) {$\delta$};\node at  (0.1,0.25)  {$\mu$};}
		-\treey{\cdlr[0.8]{o} \cdl{ol}\cdr{or}
			\node at (ol) [scale=0.6]{$\bullet$};\node at (-0.5,0.1) {$\delta$};\node at  (0.1,0.25)  {$\mu$};}
		-\treey{\cdlr[0.8]{o} \cdl{ol}\cdr{or}
			\node at (or) [scale=0.6]{$\bullet$};\node at (0.5,0.1) {$\delta$};\node at  (0.1,0.25)  {$\mu$};}\Bigg\rangle_2
		=&\pair{\delta^\ast\circ_1\mu^\ast, \delta\circ_1\mu}_2+\pair{\mu^\ast\circ_1\delta^\ast, \mu\circ_1\delta}_2\\
		=&-\delta^\ast(\delta)\mu^\ast(\mu)+\mu^\ast(\mu)\delta^\ast(\delta)=0
\end{align*}
 for any $\delta,\delta_1,\delta_2\in\Delta$. Similarly,
\begin{align*}
		&\Bigg\langle
		\treey{\cdlr[0.8]{o} \cdl{ol}\cdr{or}\zhds{o/b}
			\node at (0.3,-0.2) {$\delta^\ast$};\node at  (0.1,0.25)  {$\mu^\ast$};}
		-\treey{\cdlr[0.8]{o} \cdl{ol}\cdr{or}
			\node at (or) [scale=0.6]{$\bullet$};\node at (0.5,0.1) {$\delta^\ast$};\node at  (0.1,0.25)  {$\mu^\ast$};},
		\treey{\cdlr[0.8]{o} \cdl{ol}\cdr{or}\zhds{o/b}
			\node at (0.3,-0.2) {$\delta$};\node at  (0.1,0.25)  {$\mu$};}
		-\treey{\cdlr[0.8]{o} \cdl{ol}\cdr{or}
			\node at (ol) [scale=0.6]{$\bullet$};\node at (-0.5,0.1) {$\delta$};\node at  (0.1,0.25)  {$\mu$};}
		-\treey{\cdlr[0.8]{o} \cdl{ol}\cdr{or}
			\node at (or) [scale=0.6]{$\bullet$};\node at (0.5,0.1) {$\delta$};\node at  (0.1,0.25)  {$\mu$};}
		\Bigg\rangle_2=0,
\end{align*}
$$\Bigg\langle\treey{\cdlr{ol}\node at  (0.25,-0.15)  {$\mu^\ast$};\node at  (-0.25,0.5)  {$\mu^\ast$};}-\treey{\cdlr{or}\node at  (0.25,-0.15)  {$\mu^\ast$};\node at  (0.4,0.6)  {$\mu^\ast$};},
	\treey{\cdlr{ol}\node at  (0.25,-0.15)  {$\mu$};\node at  (-0.25,0.5)  {$\mu$};}-\treey{\cdlr{or}\node at  (0.25,-0.15)  {$\mu$};\node at  (0.4,0.6)  {$\mu$};}\Bigg\rangle_3
	=\pair{\mu^\ast\circ_1\mu^\ast,\mu\circ_1\mu}_3+\pair{\mu^\ast\circ_2\mu^\ast,\mu\circ_2\mu}_3
	=\mu^\ast(\mu)\mu^\ast(\mu)-\mu^\ast(\mu)\mu^\ast(\mu)=0.
$$
We now compute the dimension of $\stt(E)^{(2)}$:
$$\dim(\stt(E)^{(2)}_1)=|\Delta|^{2},\, \dim(\stt(E)^{(2)}_2)=3|\Delta|, \,\dim(\stt(E)^{(2)}_3)=2,$$
the dimension of the relation $\bfk R$:
$$\dim(\bfk R_1)=\frac{|\Delta|(|\Delta|-1)}{2},\, \dim(\bfk R_2)=|\Delta|, \,\dim(\bfk R_3)=1,$$
and  the dimension of the orthogonal space $R^\perp$:
$$\dim(R_1^\perp)=\frac{|\Delta|(|\Delta|+1)}{2},\, \dim(R_2^\perp)=2|\Delta|,\, \dim(R_3^\perp)=1.$$
Then the conclusion follows from the equality
$$\qquad\qquad \qquad\qquad\,
\dim(R_i^\perp)+\dim(\bfk R_i)=\dim(\stt(E)^{(2)}_i)\, \text{ for } i=1,2,3. \qquad\qquad \qquad\qquad\,\qedhere$$
\end{proof}

\subsection{Self duality of the matching compatibility}
A remarkable property of the matching compatibility is its self duality.
Let us first record an easy fact for later use.

\begin{lemma}\mlabel{lem:orfact}
Let $U, V$ be vector spaces, $W=U\oplus V$ and $\pair{-,-}:W^\ast\otimes W\rightarrow \bfk$
the natural pairing $u^*\otimes v\mapsto u^*(v)$. Suppose $X\subseteq U$ and $Y\subseteq V$.
Then
$$(X\cup Y)^\perp=\bfk \big(X^\perp|_{U^\ast}\cup Y^\perp|_{V^\ast}\big).$$
\end{lemma}
\begin{theorem}
Let $\Omega$ be a nonempty finite set.
Let $\mscr{P}=\mscr{T}(E)/\langle R\rangle $ be a finitely generated unary binary quadratic ns  operad. Then
$({\mat{\mscr{P}}_\Omega})^!=~\mat{(\mscr{P}^!)}_\Omega$.
\mlabel{thm:mdul}
\end{theorem}
\begin{proof}
Since the operad $\spp=\spp(E, R)$ is quadratic, it follows from Eqs.~(\mref{eq:pr1}) -- (\mref{eq:pr4}) that
$$R_{1,3}=R_{2,3}=R_{3,3}=R_{4}=\emptyset.$$
Abbreviate
\begin{align*}
 R_1:=R_{1,2}&=\Big\{r_{2}^n(P_k, P_\ell):=r_{1,2}^n(P_k, P_\ell)\,\big|\,{1\leq n\leq n_{1}}\Big\},\\
 R_2:=R_{2,2}&=\Big\{r_{2}^n(P_k, i):=r_{2,2}^n(P_k, i)\,\big|\,{1\leq n\leq n_{2}}\Big\},  \\
 R_3:=R_{3,2}&=\Big\{r_{2}^n(i, j):=r_{3,2}^n(i,j)\,\big|\,{1\leq n\leq n_{3}}\Big\},
\end{align*}
and
\begin{equation*}
  R_i^{\mu,\nu}:=R_{i,2}^{\mu,\nu}, \,\text{ for }\, i={1,2,3}.
\end{equation*}
Eq.~(\mref{eq:mat}) gives
\begin{equation}
\sopr{(R^\perp)}_M=\underset{\mu\neq\nu\in\Omega} {\bigcup } \Big((R^\perp)^{\mu,\nu}_1\cup (R^\perp)^{\mu,\nu}_2 \cup (R^\perp)^{\mu,\nu}_3\Big).
\mlabel{eq:mrmpp}
\end{equation}
By Definition~\mref{defn:mat1} and Theorem~\mref{lem:koszul},
\begin{eqnarray*}
(\mat{\mscr{P}}_\Omega)^!&=&\mscr{T}\Big(\big(\bigoplus\limits_{\omega\in\Omega} E_\omega\big)^\ast\Big)\Big/\Big\langle \big(\bigcup\limits_{\omega\in\Omega}R_\omega\cup \sopr{R_M}\big)^\perp\Big\rangle, \\
\mat{(\mscr{P}^!)}_\Omega&=&\mscr{T}\Big(\bigoplus\limits_{\omega\in\Omega} E^\ast_\omega\Big)\Big/\Big\langle \bigcup\limits_{\omega\in\Omega}(R^\perp)_\omega\cup \sopr{(R^\perp)}_M\Big\rangle.
\end{eqnarray*}
We identify $\mscr{T}\Big(\bigoplus\limits_{\omega\in\Omega} E^\ast_\omega\Big)$ with $\mscr{T}\Big(\big(\bigoplus\limits_{\omega\in\Omega} E_\omega\big)^\ast\Big)$ by
$\bigoplus\limits_{\omega\in\Omega} E^\ast_\omega\cong \big(\bigoplus\limits_{\omega\in\Omega} E_\omega\big)^\ast$.
Denote by $\stt_\Omega$ (resp. $\stt_{\Omegac})$ the subspace of $\mscr{T}\big(\bigoplus\limits_{\omega\in\Omega} E_\omega^\ast\big)$
generated by trees with vertices decorated by $E^\ast_\omega$ (resp. $E^\ast_{\omega_1}$,$E^\ast_{\omega_2}$,\ldots ,$E^\ast_{\omega_n}$, for $1\leq n\leq|\Omega|$), for some $\omega\in \Omega$ (resp. for ${\omega_i}\in\Omega$, not all identical).
Then
\begin{equation}
\mscr{T}\Big(\bigoplus\limits_{\omega\in\Omega} E_\omega^\ast\Big)=\stt_\Omega\oplus\stt_\Omegac.
\mlabel{eq:mtsplit}
\end{equation}
Denote by $\stt_{\Omegac,n}$ the component of $\stt_{\Omegac}$ in arity $n$.
We only need to verify
$\Big\langle (\bigcup\limits_{\omega\in\Omega}R_\omega\cup \sopr{R_M})^\perp\Big\rangle = \Big\langle \bigcup\limits_{\omega\in\Omega}(R^\perp)_\omega\cup \sopr{(R^\perp)}_M \Big\rangle$,
which will be achieved by showing the identity
\begin{equation}
 \bfk\Big( \bigcup\limits_{\omega\in\Omega}R_\omega\cup \sopr{R_M}\Big)^\perp=\bfk\Big(\bigcup\limits_{\omega\in\Omega}(R^\perp)_\omega\cup \sopr{(R^\perp)}_M \Big).
 \mlabel{eq:mequir}
\end{equation}
Now the left hand side of the equation is
\begin{eqnarray*}
&&\bfk\Big( \bigcup\limits_{\omega\in\Omega}R_\omega\cup \sopr{R_M}\Big)^\perp\\
&=&\bfk\biggr( \big(\bigcup\limits_{\omega\in\Omega}R_\omega\cup \sopr{R_M}\big)\big|_{\stt_\Omega}\sqcup\big(\bigcup\limits_{\omega\in\Omega}R_\omega\cup \sopr{R_M}\big)\big|_{\stt_\Omegac}\biggr)^\perp
 \quad\quad(\text{by Eq.~(\mref{eq:mtsplit})})\\
&=&\bfk\biggr( \big(\bigcup\limits_{\omega\in\Omega}R_\omega\cup \sopr{R_M}\big)^\perp\big|_{\stt_\Omega}\sqcup\big(\bigcup\limits_{\omega\in\Omega}R_\omega\cup \sopr{R_M}\big)^\perp\big|_{\stt_\Omegac}\biggr)
\quad\quad(\text{by Lemma}~\mref{lem:orfact})\\
&=&\bfk\biggr( \big(\bigcup\limits_{\omega\in\Omega}R_\omega\big)^\perp\big|_{\stt_\Omega}\sqcup\big( \sopr{R_M}\big)^\perp\big|_{\stt_\Omegac}\biggr)\\
&=&\bfk\biggr( \bigcup\limits_{\omega\in\Omega}\big(R^\perp)_{\omega}
\sqcup\big( \underset{\mu\neq\nu\in\Omega} {\bigcup} R^{\mu,\nu}_1\big)^\perp_1\big|_{\stt_{\Omegac,1}}
\sqcup\big( \underset{\mu\neq\nu\in\Omega} {\bigcup} R^{\mu,\nu}_2\big)^\perp_2\big|_{\stt_{\Omegac,2}}
\sqcup\big( \underset{\mu\neq\nu\in\Omega} {\bigcup} \sopr{R}^{\mu,\nu}_3\big)^\perp_3\big|_{\stt_{\Omegac,3}}\biggr)\\
&& \qquad \qquad \qquad(\text{by Eq.~(\mref{eq:mat}) and Lemma~\mref{lem:orfact}}).
\end{eqnarray*}
Applying the arity grading and Eq.~(\mref{eq:mrmpp}), then Eq.~(\mref{eq:mequir}) amounts to the three equations:
 \begin{eqnarray}
\bfk\Big(\big( \underset{\mu\neq\nu\in\Omega} {\bigcup} R^{\mu,\nu}_i\big)^\perp_i\big|_{\stt_\Omegac}\Big)
&=&\bfk \Big(\underset{\mu\neq\nu\in\Omega} {\bigcup } (R^\perp)^{\mu,\nu}_i \Big), \quad  i=1,2,3. \mlabel{eq:mkr1}
\end{eqnarray}

Denote
\begin{equation}
  (R^\perp)_1 := ~\biggr\{{r'}_1^n(P^{\dual}_k,P^{\dual}_\ell):=\sum_{1\leq k,\ell\leq t}\alpha_{k,\ell}'^{n}~~\treeyy{\cdu{o} \node  at (0,-0.2)[scale=0.6] {$\bullet$};\node at (0.3,-0.2) {$P^{\dual}_\ell$};
\node  at (0,0.2)[scale=0.6] {$\bullet$};\node at (0.3,0.2) {$P^{\dual}_k$};}\, \Big| \,1\leq n \leq n'_1\,\biggr\}.
\mlabel{eq:morr1}
\end{equation}
Then Eq.~(\mref{eq:mkr1}) for $i=1$ follows from
\begin{eqnarray*}
 \bfk\Big(\big( \underset{\mu\neq\nu\in\Omega} {\bigcup} R^{\mu,\nu}_1\big)^\perp_1\big|_{\stt_\Omegac}\Big)
 &=&\bfk \Big\{{r}_1^n(P_{\mu,k},P_{\nu,\ell})\,
\big|\,\mu\neq\nu\in\Omega,\,{1\leq n\leq n_{1}}\Big\}^\perp_1\big|_{\stt_\Omegac}\\
&=&\bfk \Big\{ {r'}_1^n(P^{\dual}_{\mu,k},P^{\dual}_{\nu,\ell})\,
\big|\,\mu\neq\nu\in\Omega,\,{1\leq n\leq n'_1}\Big\}\quad(\text{by Eqs.~}(\mref{eq:morr1}) \,\text{ and }\,(\mref{eq:rperp}))\\
&=&\bfk \Big(\underset{\mu\neq\nu\in\Omega} {\bigcup } (R^\perp)^{\mu,\nu}_1\Big).
\end{eqnarray*}
A similar argument yields Eq.~(\mref{eq:mkr1}) for $i=2,3$.
\end{proof}
\begin{coro}
\begin{enumerate}
\item $($\mcite{ZBG}$)$ The operad of matching (associative) algebras is self dual.
\item The operad of (multiple) matching Poisson algebras is self dual.
\end{enumerate}
\end{coro}
We also give some examples of self dual operads which have nontrivial unary operations.
\begin{coro}
\begin{enumerate}	
	\item Let $\spp$ be the operad with generators
$$E_1 = ~\bfk\biggr\{\treeyy[scale=0.8]{\cdu{o} \node  at (0,0) [scale=0.6]{$\bullet$};\node at (0.3,0) {$d_1$};},\,\,\treeyy[scale=0.8]{\cdu{o} \node  at (0,0) [scale=0.6]{$\bullet$};\node at (0.3,0) {$d_2$};}\biggr\},\quad
E_2 = \bfk\treey{\cdlr{o}},$$
and relations
\begin{gather*}
R_1=\left\{
\treeyy{\cdu{o} \node  at (0,-0.2) [scale=0.6]{$\bullet$};\node at (0.3,-0.2) {$d_1$};\node  at (0,0.2) [scale=0.6]{$\bullet$};\node at (0.3,0.2) {$d_1$};},\,\,
\treeyy{\cdu{o} \node  at (0,-0.2) [scale=0.6]{$\bullet$};\node at (0.3,-0.2) {$d_1$};\node  at (0,0.2) [scale=0.6]{$\bullet$};\node at (0.3,0.2) {$d_2$};}\right\},\quad
R_3=\left\{\treey{\cdlr{ol}}-\treey{\cdlr{or}}\right\},\\
~R_2=\left\{\treey{\cdlr[0.8]{o} \cdl{ol}\cdr{or}\zhds{o/b}\node at (0.3,-0.2) {$d_1$};}-\treey{\cdlr[0.8]{o} \cdl{ol}\cdr{or}\node at (ol) [scale=0.6]{$\bullet$};\node at (-0.5,0.1) {$d_1$};}
-\treey{\cdlr[0.8]{o} \cdl{ol}\cdr{or}\node at (or) [scale=0.6]{$\bullet$};\node at (0.5,0.1) {$d_1$};},\quad
\treey{\cdlr[0.8]{o} \cdl{ol}\cdr{or}\zhds{o/b}\node at (0.3,-0.2) {$d_2$};}-\treey{\cdlr[0.8]{o} \cdl{ol}\cdr{or}\node at (ol) [scale=0.6]{$\bullet$};\node at (-0.5,0.1) {$d_2$};},
\quad\treey{\cdlr[0.8]{o} \cdl{ol}\cdr{or}\zhds{o/b}\node at (0.3,-0.2) {$d_2$};}-\treey{\cdlr[0.8]{o} \cdl{ol}\cdr{or}\node at (or) [scale=0.6]{$\bullet$};\node at (0.5,0.1) {$d_2$};}\right\}.
\end{gather*}
Then $\spp$ is self dual.
\mlabel{cor:undual1}
\item For the operad $\spp$ above and any nonempty finite set $\Omega$, the operad $\mat{\spp}_\Omega$ is self dual.
\mlabel{cor:undual2}
\end{enumerate}
\mlabel{cor:undual}
\end{coro}
\begin{proof}
Item~\mref{cor:undual1} can be verified by the same computation as for Proposition~\mref{prop:kdualdda}.
Then Item~\mref{cor:undual2} follows from Theorem~\mref{thm:mdul}.
\end{proof}
We next show that taking a matching family of an operad can be obtained by taking the black square product or the white square product.
\begin{prop}
Let $\Omega$ be a nonempty finite set. Let $\spp$ be a finitely generated binary quadratic ns operad.
Then
$$\mat{\spp}_\Omega\cong \mat{\as}_\Omega\blacksquare\spp\text{ and }\,
\mat{\spp}_\Omega\cong \mat{\as}_\Omega \square \spp. $$
\mlabel{prop:maninbll}
\end{prop}
\begin{proof}
For the associative operad $\as=\stt(E)/\langle R\ \rangle$ in Example~\mref{ex:nsop}~\mref{exam:0}, by Definition~\mref{defn:comp1}, we have
\begin{align*}
\mat{\as}_\Omega
=\mscr{T}\Big(\bigoplus\limits_{\omega\in\Omega} E_\omega\Big)\Big/\Big\langle \bigcup\limits_{\omega\in\Omega}R_\omega\cup \mrr\Big\rangle
=\mscr{T}\Big(\bigoplus\limits_{\omega\in\Omega} E_\omega\Big)\Big/\Big\langle \bigcup\limits_{\omega\in\Omega}R_\omega\cup \Big(\underset{\mu\neq\nu\in\Omega} {\bigcup } R^{\mu,\nu}\Big)\Big\rangle,
\end{align*}
where the generators
$$\bigoplus\limits_{\omega\in\Omega} E_\omega=
\bigoplus\limits_{\omega\in\Omega}\bfk\treey{\cdlr{o}\node  at (0,0.35) {$\omega$}; }$$
and relations are given in Eq.~\meqref{eq:linassr}.
Let $\spp=\stt(F)/\langle S \rangle$  with
\begin{equation*}
	F :=F_2:=\bfk\left\{
\treey{\cdlr{o}\node  at (0,0.25) {$1$};},
\treey{\cdlr{o}\node  at (0,0.25) {$2$};},~... ~,
\treey{\cdlr{o}\node  at (0,0.25) {$s$};}
\right\}
\end{equation*}
and
\begin{equation*}
S:=S_{3,2}:= ~\biggr\{\sum_{1\leq i,j\leq s}\ka{i}{j}{1}
\treey{\cdlr{ol}
\foreach \i/\j in {ol/$i$,o/$j$} {\node[above] at (\i) {\j};}}
+\ka{i}{j}{2}\treey{\cdlr{or}
\foreach \i/\j in {or/$j$,o/$i$} {\node[above] at (\i) {\j};}}
\quad \Big| \,{1\leq d\leq s}\biggr\}.
\end{equation*}
Recall from Eq.~\meqref{eq:wqpiso} that
$$\Big(\bigoplus\limits_{\omega\in\Omega} E_\omega \Big)\otimes F
\cong\bigoplus\limits_{\omega\in\Omega} F_\omega .$$
By Definitions~\mref{defn:manin}~\mref{item:maninb}, we obtain
\begin{eqnarray*}
&&\Big(\bigcup\limits_{\omega\in\Omega}R_\omega\cup\underset{\mu\neq\nu\in\Omega} {\bigcup } R^{\mu,\nu}\Big)\blacksquare S\\
&=&\bigcup\limits_{\omega\in\Omega}\big(R_\omega\blacksquare S\big)\cup\underset{\mu\neq\nu\in\Omega} {\bigcup }\Big( R^{\mu,\nu}\blacksquare S\Big)\\
&=&\bigcup\limits_{\omega\in\Omega}\biggr\{\treey{\cdlr{ol}\foreach \i/\j in {ol/$\omega$,o/$\omega$} {\node [above] at (\i) {\j};}}
 -\treey{\cdlr{or}\foreach \i/\j in {or/$\omega$,o/$\omega$} {\node[above] at (\i) {\j};}}\biggr\}
 \blacksquare \biggr\{\sum_{1\leq i,j\leq s}\ka{i}{j}{1}
\treey{\cdlr{ol}
\foreach \i/\j in {ol/$i$,o/$j$} {\node[above] at (\i) {\j};}}
+\ka{i}{j}{2}\treey{\cdlr{or}
\foreach \i/\j in {or/$j$,o/$i$} {\node[above] at (\i) {\j};}}
\quad \Big| \,{1\leq d\leq s}\biggr\}\\
&& \cup \underset{\mu\neq\nu\in\Omega} {\bigcup }\biggr\{
\treey{\cdlr{ol}\foreach \i/\j in {ol/$\mu$,o/$\nu$} {\node [above] at (\i) {\j};}}
 -\treey{\cdlr{or}\foreach \i/\j in {or/$\nu$,o/$\mu$} {\node[above] at (\i) {\j};}}\biggr\}
 \blacksquare
 \biggr\{\sum_{1\leq i,j\leq s}\ka{i}{j}{1}
\treey{\cdlr{ol}
\foreach \i/\j in {ol/$i$,o/$j$} {\node[above] at (\i) {\j};}}
+\ka{i}{j}{2}\treey{\cdlr{or}
\foreach \i/\j in {or/$j$,o/$i$} {\node[above] at (\i) {\j};}}
\quad \Big| \,{1\leq d\leq s}\biggr\}\quad(\text{by Eq.}~\meqref{eq:linassr})\\
&\cong&\bigcup\limits_{\omega\in\Omega}\biggr\{\sum_{1\leq i,j\leq s}\ka{i}{j}{1}
\treey{\cdlr{ol}
\foreach \i/\j in {ol/$i_\omega$,o/$j_\omega$} {\node[above] at (\i) {\j};}}
+\ka{i}{j}{2}\treey{\cdlr{or}
\foreach \i/\j in {or/$j_\omega$,o/$i_\omega$} {\node[above] at (\i) {\j};}}
\quad \Big| \,{1\leq d\leq s}\biggr\}\\
&& \cup \underset{\mu\neq\nu\in\Omega} {\bigcup }\biggr\{\sum_{1\leq i,j\leq s}\Big(\ka{i}{j}{1}
\treey{\cdlr{ol}\foreach \i/\j in {ol/$i_\mu$,o/$j_\nu$} {\node [above] at (\i) {\j};}}+\ka{i}{j}{2}\treey{\cdlr{or}\foreach \i/\j in {or/$j_\nu$,o/$i_\mu$} {\node[above] at (\i) {\j};}}\Big)
\quad \Big| {1\leq d\leq s}\biggr\}\quad\quad(\text{by Eqs.}~\meqref{eq:maninb} \text{ and }~\meqref{eq:wqpiso})\\
&=& \bigcup\limits_{\omega\in\Omega}S_{\omega,3,2}\cup\underset{\mu\neq\nu\in\Omega} {\bigcup } S_{3,2}^{\mu,\nu} \quad\quad(\text{by Eqs.}~\meqref{eq:pr3b}\text{ and }~\meqref{eq:spr123}).
\end{eqnarray*}
Therefore
\begin{eqnarray*}
 \mat{\as}_\Omega\blacksquare\spp
&=&\stt\Big(\big(\bigoplus\limits_{\omega\in\Omega} E_\omega \big)\otimes F\Big)\Big/\Big\langle \Big(\bigcup\limits_{\omega\in\Omega}R_\omega\cup\underset{\mu\neq\nu\in\Omega} {\bigcup } R^{\mu,\nu}\Big)\blacksquare S\Big\rangle\\
&\cong&\stt\Big(\bigoplus\limits_{\omega\in\Omega} F_\omega \Big)\Big/\Big\langle \bigcup\limits_{\omega\in\Omega}S_{\omega,3,2}\cup\underset{\mu\neq\nu\in\Omega} {\bigcup } S_{3,2}^{\mu,\nu}\Big\rangle\\
&=&\mat{\spp}_\Omega.
\end{eqnarray*}
To prove the second isomorphism, we recall the duality between the black square products and white square products for any finitely generated binary ns operads $\spp$ and $\mscr{Q}$~\mcite{Val}:
\begin{equation*}
	(\spp\blacksquare\mscr{Q})^!=\spp^!\square\mscr{Q}^!.
\end{equation*}
Applying it to the first isomorphism and utilizing Theorem~\mref{thm:mdul}, we obtain
$$ \mat{(\spp^!)}_\Omega\cong (\mat{\spp}_\Omega)^! \cong (\mat{\as}_\Omega\blacksquare \spp)^!
\cong \mat{\as}_\Omega \square \spp^!,$$
hence the result after replacing $\spp^!$ by $\spp$.
\end{proof}

\section{Total compatibility and Koszul duality}\mlabel{sec:totcom}
In this section, we study the total compatibility. It is stronger than the matching compatibility and is in duality with the linear compatibility.

\subsection{Totally compatible operads}
Let $\spp$ be a unary binary \qc ns operad.
Recall from Eqs.~(\mref{eq:pr1}) -- \meqref{eq:pr4} that the relations $R$ of $\spp$. For $r=\sum \alpha_t t\in R$ with $\alpha_t\in\bfk$, the {\bf support} ${\rm Supp}(r)$ of $r$ is defined to be the set of decorated tree $t$ with $\alpha_t\neq0$. For $S\subseteq R$, denote
$${\rm Supp}(S):=\bigcup_{r\in S} {\rm Supp}(r).$$
Let $\Omega$ be a nonempty set.
\nc\rtmn{\mu,\nu}
For $\rtmn\in\Omega$, we define
\begin{eqnarray*}
\tra{R^{\rtmn}}&:=& \Big\{t_{1,2} (P_{\mu,k},P_{\nu,\ell})-t_{1,2} (P_{\nu,k},P_{\mu,\ell})\,\big| \,{t_{1,2} (P_{k},P_{\ell}) \in {\rm Supp}(R_{1,2})}\Big\},\\
\traa{R^{\rtmn}}&:=& \Big\{t_{1,3} (P_{\mu,k},P_{\nu,\ell},P_{\mu,m})-t_{1,3} (P_{\nu,k},P_{\mu,\ell},P_{\mu,m}),\\
&&t_{1,3} (P_{\mu,k},P_{\nu,\ell},P_{\mu,m})-t_{1,3} (P_{\mu,k},P_{\mu,\ell},P_{\nu,m})\,\big| \,{t_{1,3} (P_{k},P_{\ell},P_{\omega}) \in {\rm Supp}(R_{1,3})}\Big\},\\
\trbb{R^{\rtmn}}&:= &\Big\{t_{2,2} (P_{\mu,k},i_\nu)-t_{2,2} (P_{\nu,k},i_\mu)\,\big| \,{t_{2,2} (P_{k},i)\in {\rm Supp}(R_{2,2})}\Big\},\\
\trb{R^{\rtmn}}&:=& \Big\{t_{2,3} (P_{\mu,k}, {P_{\nu,\ell}},i_\mu)-t_{2,3} (P_{\nu,k}, {P_{\mu,\ell}},i_\mu),\,\\
&&t_{2,3} (P_{\mu,k}, {P_{\nu,\ell}},i_\mu)-t_{2,3} (P_{\mu,k}, {P_{\mu,\ell}},i_\nu)\,\big| \,{t_{2,3} (P_{k},P_{\ell},i)\in {\rm Supp}(R_{2,3})}\Big\},\\
\trc{R^{\rtmn}}&:=& \Big\{t_{3,2} (i_{\mu},j_{\nu})-t_{3,2} (i_{\nu},j_{\mu})\,\big| \,{t_{3,2} (i,j)\in {\rm Supp}(R_{3,2})}\Big\},\\
\trcc{R^{\rtmn}}&:=& \Big\{t_{3,3} (P_{\mu,k},i_{\nu},j_{\mu})-t_{3,3} (P_{\nu,k},i_{\mu},j_{\mu}),\\
&&t_{3,3} (P_{\mu,k},i_{\nu},j_{\mu})-t_{3,3} (P_{\mu,k},i_{\mu},j_{\nu})\,\big| \,{t_{3,3} (P_{k},i,j)\in {\rm Supp}(R_{3,3})}\Big\},\\
\trd{R^{\rtmn}}&:=& \Big\{t_{4} (i_{\mu},j_{\nu},p_{\mu})-t_{4} (i_{\nu},j_{\mu},p_{\mu}),
t_{4} (i_{\mu},j_{\nu},p_{\mu})-t_{4} (i_{\mu},j_{\mu},p_{\nu})\,\big| \,{t_{4} (i,j,p)\in {\rm Supp}(R_{4})}\Big\}.
\end{eqnarray*}
Set
\begin{equation}
\trr:=\bigcup_{\mu\neq\nu\in\Omega\atop i\in\{1,2,3\}}\big(\tri{R^{\rtmn}}\cup \trii{R^{\rtmn}}\cup \trd{R^{\rtmn}}\big).
\mlabel{eq:tr123}
\end{equation}

Roughly speaking, these relations means that the decorated trees appearing in the relations of $\spp$ can carry arbitrary decorations in $\Omega$.

\begin{exam}
Let $\Omega=\{1,2\}$. For the associative operad $\as$, we have
$${\rm Supp}(R_{3,2})=\Big\{\treey{\cdlr{ol}},\treey{\cdlr{or}}\Big\}, $$
\begin{equation*}
\trc{R^{1,2}}
=\Big\{
\treey{\cdlr{ol}\foreach \i/\j in {ol/$1$,o/$2$} {\node [above] at (\i) {\j};}}
 -\treey{\cdlr{ol}\foreach \i/\j in {ol/$2$,o/$1$} {\node [above] at (\i) {\j};}},\,
 \treey{\cdlr{or}\foreach \i/\j in {or/$2$,o/$1$} {\node[above] at (\i) {\j};}}
 -\treey{\cdlr{or}\foreach \i/\j in {or/$1$,o/$2$} {\node[above] at (\i) {\j};}}
\Big\} = - \trc{R^{2,1}}.
\end{equation*}
\end{exam}
Now we give our last compatibility condition for operads. Recall that $\mrr$ is given in Eq.~(\mref{eq:mat}).
\begin{defn}
Let $\Omega$ be a nonempty set. Let $\mscr{P} = \stt(E)/\langle R\rangle$ be a unary binary \qc ns operad.
We call
\begin{equation*}
 \tot{\mscr{P}}_\Omega :=\mscr{T}\Big(\bigoplus\limits_{\omega\in\Omega} E_\omega\Big)\Big/\Big\langle \bigcup\limits_{\omega\in\Omega}R_\omega\cup \mrr\cup\trr\Big\rangle
\end{equation*}
the \textbf{totally compatible operad } of $\mscr{P}$ with parameter $\Omega$.
\mlabel{defn:totcomp1}
\end{defn}
\begin{exam}
Let $\Omega=\{1,2\}$ and $\spp$ be the operad of Rota-Baxter algebras with the unary operation $P$ and binary operation $\bullet$. Then a $\tot{\mscr{P}}_\Omega$-algebra is a vector space $A$ with associative multiplications $\multa, \multb$ and Rota-Baxter operators $P_1, P_2$ satisfying the additional conditions
\begin{equation*}
\prl{\multa}{\multb}=\prl{\multb}{\multa}=\prr{\multa}{\multb}=\prr{\multb}{\multa}\end{equation*}
and
\begin{eqnarray*}
&\ra{\operi}{\multj}{\operi}=\operi(\operi(x)\multj y)+\operi(x\multj \operi(y)),&\\
&\ra{\operi}{\multi}{\operj}=\ra{\operj}{\multi}{\operi}=\ra{\operi}{\multj}{\operi},&\\
&\operj(\operi(x)\multi y)=\operi(\operj(x)\multi y)=\operi(\operi(x)\multj y),&\\
&\operi(x\multi \operj(y))=\operj(x\multi \operi(y))=\operi(x\multj \operi(y)),&\\
\end{eqnarray*}
for $i\neq j\in \{1,2\}$. Here the last three lines are relations $\trb{R^{\mu,\nu}}$, for $\mu,\nu\in\Omega$.
\mlabel{ex:rbtot}
\end{exam}

Since $\tot{\spp}$ has the extra conditions $\trr$ beyond $\mat{\spp}$, we have
\begin{prop}
Let $\spp$ be a \ubqco.
Then there is an epimorphism of ns operads
$$\mat{\spp}_\Omega \longrightarrow \tot{\spp}_\Omega.$$
In other words, a $\tot{\spp}_\Omega$-algebra is a $\mat{\spp}_\Omega$-algebra.
\mlabel{prop:totmat0}
\end{prop}

\subsection{Total compatibility, Koszul duality and Manin white square products}
We first show that the total compatibility is in Koszul dual to the linear compatibility.
\begin{theorem}
Let $\Omega$ be a nonempty finite set.
Let $\mscr{P}=\mscr{T}(E)/\langle R\rangle $  be a finitely generated unary binary quadratic ns  operad. Then
$({\lin{\mscr{P}}_\Omega})^!=~\tot{(\mscr{P}^!)}_\Omega$ and $(~\tot{\mscr{P}}_\Omega)^!=\lin{(\mscr{P}^!)}_\Omega$.
\mlabel{thm:dul}
\end{theorem}
\begin{proof}
We first prove $({\lin{\mscr{P}}_\Omega})^!=~\tot{(\mscr{P}^!)}_\Omega$.
Since the operad $\spp=\spp(E, R)$ is quadratic, it follows from Eqs.~(\mref{eq:pr1}) -- (\mref{eq:pr4}) that
$$R_{1,3}=R_{2,3}=R_{3,3}=R_{4}=\emptyset.$$
Denote
\begin{align*}
 R_1:=R_{1,2}&=\Big\{r_{2}^n(P_k, P_\ell):=r_{1,2}^n(P_k, P_\ell)\,\big|\,{1\leq n\leq n_{1}}\Big\},\\
 R_2:=R_{2,2}&=\Big\{r_{2}^n(P_k, i):=r_{2,2}^n(P_k, i)\,\big|\,{1\leq n\leq n_{2}}\Big\},  \\
 R_3:=R_{3,2}&=\Big\{r_{2}^n(i, j):=r_{3,2}^n(i,j)\,\big|\,{1\leq n\leq n_{3}}\Big\},
\end{align*}
and
\begin{equation*}
R_i^{\mu,\nu}:=R_{i,2}^{\mu,\nu},\, R_{T, i}^{\mu,\nu}:=R_{T,i,2}^{\mu,\nu} \,\text{ for }\, i={1,2,3}.
\end{equation*}
By Eq.~(\mref{eq:mat}), we have
\begin{equation}
\sopr{(R^\perp)}_M=\underset{\mu\neq\nu\in\Omega} {\bigcup } \Big((R^\perp)^{\mu,\nu}_1\cup (R^\perp)^{\mu,\nu}_2 \cup (R^\perp)^{\mu,\nu}_3\Big).
\mlabel{eq:rmpp}
\end{equation}
By Definitions~\mref{defn:comp1} and~\mref{defn:totcomp1}, and Theorem~\mref{lem:koszul}, we obtain
\begin{eqnarray*}
(\lin{\mscr{P}}_\Omega)^!&=&\mscr{T}\Big(\big(\bigoplus\limits_{\omega\in\Omega} E_\omega\big)^\ast\Big)\Big/\Big\langle \big(\bigcup\limits_{\omega\in\Omega}R_\omega\cup \sopr{\lrr}\big)^\perp\Big\rangle,\\
\tot{(\mscr{P}^!)}_\Omega&=&\mscr{T}\Big(\bigoplus\limits_{\omega\in\Omega} E^\ast_\omega\Big)\Big/\Big\langle \bigcup\limits_{\omega\in\Omega}(R^\perp)_\omega\cup \sopr{(R^\perp)}_M\cup \sopr{(R^\perp)}_T \Big\rangle.
\end{eqnarray*}

Identify $\mscr{T}\Big(\bigoplus\limits_{\omega\in\Omega} E^\ast_\omega\Big)$ with $\mscr{T}\Big(\big(\bigoplus\limits_{\omega\in\Omega} E_\omega\big)^\ast\Big)$ by
$\bigoplus\limits_{\omega\in\Omega} E^\ast_\omega\cong \big(\bigoplus\limits_{\omega\in\Omega} E_\omega\big)^\ast$.
Denote by $\stt_\Omega$ (resp. $\stt_{\Omegac}$) the subspace of $\mscr{T}\big(\bigoplus\limits_{\omega\in\Omega} E_\omega^\ast\big)$,
spanned by trees with vertices decorated by $E^\ast_\omega$ (resp. by $E^\ast_{\omega_1}$,$E^\ast_{\omega_2}$,\ldots ,$E^\ast_{\omega_n}$, for $1\leq n\leq|\Omega|$), for some $\omega\in \Omega$ (resp. for ${\omega_i}\in\Omega$ not all identical).
Then
\begin{equation}
\mscr{T}\Big(\bigoplus\limits_{\omega\in\Omega} E_\omega^\ast\Big)=\stt_\Omega\oplus\stt_\Omegac.
\mlabel{eq:tsplit}
\end{equation}
Denote by $\stt_{\Omegac,n}$ the component of $\stt_{\Omegac}$ in arity $n$. We only need to show
$$\langle (\bigcup\limits_{\omega\in\Omega}R_\omega\cup \sopr{\lrr})^\perp\rangle=\langle \bigcup\limits_{\omega\in\Omega}(R^\perp)_\omega\cup \sopr{(R^\perp)}_M\cup \sopr{(R^\perp)}_T \rangle,$$
which follows from the equality
\begin{equation}
 \bfk\Big( \bigcup\limits_{\omega\in\Omega}R_\omega\cup \sopr{\lrr}\Big)^\perp=\bfk\Big(\bigcup\limits_{\omega\in\Omega}(R^\perp)_\omega\cup \sopr{(R^\perp)}_M\cup \sopr{(R^\perp)}_T \Big).
 \mlabel{eq:equir}
\end{equation}
The left hand side of the equality is
\begin{eqnarray*}
&&\bfk\Big( \bigcup\limits_{\omega\in\Omega}R_\omega\cup \sopr{\lrr}\Big)^\perp\\
&=&\bfk\biggr( \big(\bigcup\limits_{\omega\in\Omega}R_\omega\cup \sopr{\lrr}\big)\big|_{\stt_\Omega}\sqcup\big(\bigcup\limits_{\omega\in\Omega}R_\omega\cup \sopr{\lrr}\big)\big|_{\stt_\Omegac}\biggr)^\perp\quad \quad \quad\quad\quad(\text{by Eq.~(\mref{eq:tsplit})})\\
&=&\bfk\biggr( \big(\bigcup\limits_{\omega\in\Omega}R_\omega\cup \sopr{\lrr}\big)^\perp\big|_{\stt_\Omega}\sqcup\big(\bigcup\limits_{\omega\in\Omega}R_\omega\cup \sopr{\lrr}\big)^\perp\big|_{\stt_\Omegac}\biggr)
\quad\quad(\text{by Lemma}~\mref{lem:orfact})\\
&=&\bfk\biggr( \big(\bigcup\limits_{\omega\in\Omega}R_\omega\big)^\perp\big|_{\stt_\Omega}\sqcup\big( \sopr{\lrr}\big)^\perp\big|_{\stt_\Omegac}\biggr)\\
&=&\bfk\biggr( \bigcup\limits_{\omega\in\Omega}\big(R^\perp)_{\omega}
\sqcup\big( \underset{\mu\neq\nu\in\Omega} {\bigcup} (R^{\mu,\nu}_1\add R^{\nu,\mu}_1)\big)^\perp_1\big|_{\stt_{\Omegac,1}}
\sqcup\big( \underset{\mu\neq\nu\in\Omega} {\bigcup} (R^{\mu,\nu}_2\add R^{\nu,\mu}_2)\big)^\perp_2\big|_{\stt_{\Omegac,2}}\\
&&\sqcup\big( \underset{\mu\neq\nu\in\Omega} {\bigcup} (\sopr{R}^{\mu,\nu}_3\add \sopr{R}^{\nu,\mu}_3)\big)^\perp_3\big|_{\stt_{\Omegac,3}}\biggr)\quad\quad(\text{by Eq.}~\meqref{eq:lin1} \text{ and Lemma}~\mref{lem:orfact}).
\end{eqnarray*}
Applying the arity grading, and Eqs.~(\mref{eq:tr123}) and~(\mref{eq:rmpp}), one finds that Eq.~(\mref{eq:equir}) is equivalent to the equations:
 \begin{eqnarray}
\bfk\Big(\big( \underset{\mu\neq\nu\in\Omega} {\bigcup} (R^{\mu,\nu}_i\add R^{\nu,\mu}_i)\big)^\perp_i\big|_{\stt_\Omegac}\Big)
&=&\bfk \Big(\underset{\mu\neq\nu\in\Omega} {\bigcup } \big((R^\perp)^{\mu,\nu}_i\cup {(R^\perp)}_{T,i}^{\mu,\nu}\big) \Big),  \quad i=1,2,3.\mlabel{eq:kr1}
\end{eqnarray}

First denote
\begin{equation}
  (R^\perp)_1 = ~\biggr\{\sum_{1\leq k,\ell\leq t}\alpha_{k,\ell}'^{n}~~\treeyy{\cdu{o} \node  at (0,-0.2)[scale=0.6] {$\bullet$};\node at (0.3,-0.2) {$P^{\dual}_l$};
\node  at (0,0.2)[scale=0.6] {$\bullet$};\node at (0.3,0.2) {$P^{\dual}_k$};}\, \Big| \,1\leq n \leq n'_1\,\biggr\}.
\mlabel{eq:orr1}
\end{equation}
Then Eq.~(\mref{eq:kr1}) for $i=1$ follows from
\begin{eqnarray*}
 &&\bfk\Big(\big( \underset{\mu\neq\nu\in\Omega} {\bigcup} (R^{\mu,\nu}_1\add R^{\nu,\mu}_1)\big)^\perp_1\big|_{\stt_\Omegac}\Big)\\
&=&\bfk \biggr\{\sum_{1\leq k,\ell\leq t}\alpha_{k,\ell}^{n}~~
\big(\treeyy{\cdu{o} \node  at (0,-0.2)[scale=0.6] {$\bullet$};\node at (0.35,-0.2) {$P_{\nu,\ell}$};
\node  at (0,0.2)[scale=0.6] {$\bullet$};\node at (0.35,0.2) {$P_{\mu,k}$};}+\treeyy{\cdu{o} \node  at (0,-0.2)[scale=0.6] {$\bullet$};\node at (0.35,0.2) {$P_{\nu,k}$};
\node  at (0,0.2)[scale=0.6] {$\bullet$};\node at (0.35,-0.2) {$P_{\mu,\ell}$};}\big)
\quad \Big| \,\mu\neq\nu\in\Omega,\,{1\leq n\leq n_{1}}\biggr\}^\perp_1\big|_{\stt_\Omegac}\\
&=&\bfk \biggr\{ \sum_{1\leq k,\ell\leq t}\alpha_{k,\ell}'^{n}~~
\treeyy{\cdu{o} \node  at (0,-0.2)[scale=0.6] {$\bullet$};\node at (0.35,-0.2) {$P^\ast_{\nu,\ell}$};
\node  at (0,0.2)[scale=0.6] {$\bullet$};\node at (0.35,0.2) {$P^\ast_{\mu,k}$};}, \,
\treeyy{\cdu{o} \node  at (0,-0.2)[scale=0.6] {$\bullet$};\node at (0.35,-0.2) {$P^\ast_{\nu,\ell}$};
\node  at (0,0.2)[scale=0.6] {$\bullet$};\node at (0.35,0.2) {$P^\ast_{\mu,k}$};}-\treeyy{\cdu{o} \node  at (0,-0.2)[scale=0.6] {$\bullet$};\node at (0.35,-0.2) {$P^\ast_{\mu,\ell}$};
\node  at (0,0.2)[scale=0.6] {$\bullet$};\node at (0.35,0.2) {$P^\ast_{\nu,k}$};}
\quad \Big| \,\mu\neq\nu\in\Omega,\,{1\leq n\leq n'_1}\biggr\}\quad(\text{by Eqs.~}(\mref{eq:orr1}) \,\text{ and }\,(\mref{eq:rperp}))\\
&=&\bfk \biggr(\biggr\{ \sum_{1\leq k,\ell\leq t}\alpha_{k,\ell}'^{n}~~
\treeyy{\cdu{o} \node  at (0,-0.2)[scale=0.6] {$\bullet$};\node at (0.35,-0.2) {$P^\ast_{\nu,\ell}$};
\node  at (0,0.2)[scale=0.6] {$\bullet$};\node at (0.35,0.2) {$P^\ast_{\mu,k}$};}
\quad \Big| \,\mu\neq\nu\in\Omega,\,{1\leq n\leq n'_1}\biggr\}\cup\biggr\{\treeyy{\cdu{o} \node  at (0,-0.2)[scale=0.6] {$\bullet$};\node at (0.35,-0.2) {$P^\ast_{\nu,\ell}$};
\node  at (0,0.2)[scale=0.6] {$\bullet$};\node at (0.35,0.2) {$P^\ast_{\mu,k}$};}-\treeyy{\cdu{o} \node  at (0,-0.2)[scale=0.6] {$\bullet$};\node at (0.35,-0.2) {$P^\ast_{\mu,\ell}$};
\node  at (0,0.2)[scale=0.6] {$\bullet$};\node at (0.35,0.2) {$P^\ast_{\nu,k}$};}\quad \Big| \,\mu\neq\nu\in\Omega\biggr\}\biggr)\\
&=&\bfk \Big(\underset{\mu\neq\nu\in\Omega} {\bigcup }\big( (R^\perp)^{\mu,\nu}_1\cup {(R^\perp)}_{T,1}^{\mu,\nu} \big)\Big).
\end{eqnarray*}
We similarly verify Eq.~(\mref{eq:kr1}) for $i=2, 3$.
This completes the proof of $({\lin{\mscr{P}}_\Omega})^!=~\tot{(\mscr{P}^!)}_\Omega$.
Since $(\spp^!)^!=\spp$ for any quadratic operad $\spp$, we have $\lin{\spp}_\Omega=({(\lin{\spp}_\Omega)^!})^!=(\tot{(\mscr{P}^!)}_\Omega)^!$
and so  $\lin{(\spp^!)} =(\tot{\spp})^!.$
\end{proof}

We finally show that taking the total compatibility of an operad amounts to taking the Manin white square product with the operad of totally compatible associative algebras.

\begin{coro} Let $\spp$ be a \ubqco and $\Omega$ a nonempty set. Then
	$\tot{\spp}_\Omega\cong \tot{\as}_\Omega \square \spp.$
\mlabel{cor:totalwhite}
\end{coro}

\begin{proof}
The isomorphism follows from taking the Koszul dual of the isomorphism in Proposition~\mref{prop:maninbl} and applying Theorem~\mref{thm:dul}:
$$ \tot{\spp}_\Omega = \tot{((\spp^!)^!)}_\Omega = \big(\lin{(\spp^!)}_\Omega\big)^! \cong  (\lin{\as}_\Omega\blacksquare \spp^!)^!
= \big(\lin{\as}_\Omega\big)^! \square (\spp^!)^!=\tot{(\as^!)}_\Omega \square \spp=\tot{\as}_\Omega \square \spp.  \qedhere$$
\end{proof}

\smallskip

\kong{\bf Acknowledgments}:
This work is supported by the National Natural Science Foundation of China (Grant No. 11771190 and 12071191), the Natural Science Foundation of Gansu Province (Grant No. 20JR5RA249) and the Natural Science Foundation of Shandong Province (ZR2020MA002).

\end{document}